\documentclass[11pt,a4paper]{article}
\usepackage[latin1]{inputenc} 
\usepackage[english]{babel}
\usepackage[T1]{fontenc} 
\usepackage[all]{xy}
\usepackage{fancyhdr}
\usepackage[dvips]{geometry}
\usepackage{indentfirst,amsfonts,amsmath,amsbsy,amsthm,amssymb,amscd}
\usepackage{enumerate}
\usepackage[normalem]{ulem}  
\usepackage{makeidx} 
\usepackage{color}

\usepackage{ mathrsfs }
\usepackage{dsfont}
\usepackage{hyperref}
\usepackage{calligra}
\usepackage[affil-it]{authblk}
\usepackage{graphicx}

\usepackage{epstopdf}

\catcode`@=11 \@addtoreset{equation}{section} \catcode`@=12

\newtheorem{teo}{Theorem}[section]
\newtheorem{lema}[teo]{Lemma}
\newtheorem{remark}[teo]{Remark}
\newtheorem{corolario}[teo]{Corollary}
\newtheorem{prop}[teo]{Proposition}
\newtheorem{definition}[teo]{Definition}

\def\XXint#1#2#3{{\setbox0=\hbox{$#1{#2#3}{\int}$ }
\vcenter{\hbox{$#2#3$ }}\kern-.6\wd0}}

\newcommand{\R}{\mathbb{R}}

\usepackage{epstopdf}

\usepackage{fancyhdr} 

\begin{document}
\title{Linear nonlocal diffusion problems in metric measure spaces}

\author{An\'{i}bal Rodr\'{i}guez-Bernal %
  \thanks{Partially supported by the projects MTM2012-31298 and Grupo CADEDIF GR58/08, Grupo 920894. email: arober@mat.ucm.es
  }}
\affil{Departamento de Matemática Aplicada\\
Universidad Complutense de Madrid, 28040, Madrid,  \&\\
Instituto de Ciencias Matem\'aticas, CSIC-UAM-UC3M-UCM, Madrid}
\author{Silvia Sastre-G\'{o}mez  %
  ${}^{*,}$\thanks{Partially supported by the FPU grant from the Spanish Ministerio de Educación and the projects MTM2012-31298 and Grupo CADEDIF GR58/08, Grupo 920894. email: silviasastre@mat.ucm.es.}}
\affil{Departamento de Matemática Aplicada\\Universidad Complutense de Madrid, 28040, Madrid}
\date{ \today}
%
%
\maketitle


\begin{abstract}
	 The aim of this paper is to provide a comprehensive study of
         some  linear nonlocal diffusion  problems  in metric measure
         spaces. These include, for example, open subsets in $\R^N$, 
	 graphs, manifolds,  multi-structures or some fractal
         sets. For this,  we study regularity,  compactness,
         positiveness and the spectrum of the stationary nonlocal operator.  
	 Then we study the solutions of linear evolution  nonlocal
         diffusion problems, with emphasis in  similarities  and
         differences with the standard heat equation in smooth domains. 
	 In particular prove weak and strong maximum principles and
         describe the asymptotic behaviour   using spectral methods. 
\end{abstract}


\section{Introduction}

Diffusion is the natural process by which some magnitude (heat or matter, for
example) is transported from one part of a system to another as a
result of random molecular motions. As such, diffusion has a prominent
role in distinct fields such as biology, thermodynamics and even
economics. 

In smooth media (e.g. an open region in the Euclidean space or  a smooth
manifold) classical diffusion models include differential operator
such as the Laplacian and diffusion problems are usually described in
terms of partial differential equations \cite{Crank}.  As the real world
is nonsmooth, in the last decade there has been great effort in developing similar
techniques and structures from the realm of differential equations to
analyze diffusion processes in nonsmooth media, including some fractal
like sets, see e.g. \cite{Havlin,Strichartz,kigami_book}. 

There is another approach, however, that allows to describe and model
diffusion processes by means of nonlocal models, see e.g.
\cite{Rossi_libro}, which we apply here in smooth and nonsmooth media. Assume
then that $(\Omega,\mu)$ is a measure space  and  $u(x,t)$ is the
density of some  population at the point $x\in \Omega$ at time
$t$. Also assume  $J(x,y)$ is a
nonnegative function defined in $\Omega\times\Omega$, that represents
the density of probability of a member of that population to jump from a location $y$ to $x$.
Hence $\int_{\Omega}J(y,x)dy=1$ for all $x\in \Omega$. Then $\int_{\Omega}J(x,y)u (y,t)dy$ is the
rate at which the individuals arrive to location $x$ from all other
locations $y\in \Omega$. On the other hand,
$-\int_{\Omega}J(y,x)dy\,u(x,t) = -u(x,t)$ is the rate at which the
individuals are leaving from location $x$ to all other locations $y\in
\Omega $.  Then, the time evolution of the population $u$ in $\Omega$
can be written as 
\begin{equation}\label{introduction}
\left\{
\begin{array}{lll}
u_t(x,t)&\displaystyle=\int_{\Omega}J(x,y)u(y,t)dy-u(x,t), &x\in\Omega, \\
u(x,0)&=u_0(x), & x\in\Omega.
\end{array}
\right.
\end{equation}
where $u_0$ is the initial distribution of the population.  This
problem and variations of it have been previously used to model
diffusion processes, in \cite{Rossi_libro}, \cite{Rossi}, \cite{fife},
and \cite{hutson}, for example, with $\Omega$ and open set in
$\R^{N}$. However, nonlocal diffusion models like (\ref{introduction})
can be naturally defined in  measure spaces, since we just need
to consider the density of probability of jumping from a location $x$
in $\Omega$ to a location $y$ in $\Omega$, given by the function
$J(x,y)$. 
This allows us studying   diffusion processes  in very different type of spaces, 
like: graphs, (which are used to model complicated structures in chemistry, molecular 
biology or electronics, or they can also represent basic electric circuits into digital 
computers), compact manifolds,  multi-structures  composed by several compact sets 
with different dimensions, (for example a dumbbell domain), or even some fractal sets as 
the Sierpinski gasket, \cite{Havlin, Kigami, Strichartz}. Some of this
spaces are introduced in Section \ref{sec:metr-meas-spac}.  

Since it is always convenient to speak about continuity, in  this
work, we consider problems like (\ref{introduction}) defined in
metric measure spaces, $(\Omega,\mu, d)$, which are defined as
follows. For more information see \cite{rudin}.  
 \begin{definition}\label{def_measure}
A {\bf metric measure space} $( \Omega,\mu, d)$ is a metric space 
$( \Omega,d)$  with a $\sigma$-finite, regular, and complete Borel measure 
$\mu$ in $ \Omega $, and  that  associates a finite  positive measure to the 
balls of $ \Omega $.
\end{definition} 

In this context, we take  $X=L^p(\Omega)$, 
$1\le p\le \infty$, or $X=\mathcal{C}_b(\Omega)$ and  consider
nonlocal diffusion problems of the form 
\begin{equation}\label{model_intro}
	\left\{
	\begin{array}{rl}
	u_t(x,t) & =K_J u(x,t) - h(x) u(x,t),\quad x\in\Omega,\, t>0,\\
	u(x,t_0)  & =u_0(x), \quad x\in\Omega,
	\end{array}
	\right.
\end{equation}
where $u_0\in X$,   $h\in L^{\infty}(\Omega)$ or in  $\mathcal
 {C} _b(\Omega)$, and the nonlocal diffusion operator  
$K_Ju$ is given by 
$$ 
K_Ju(x,t)=\int_{\Omega}J(x,y)u(y,t)dy . 
$$  
We will not assume, unless 
 otherwise made explicit, that  $\int_{\Omega}J(x,y)dy=1$.  A
 particular case which we will pay attention below is when $h(x) =
 \int_{\Omega}J(x,y)dy$. 

One of our main goals in this paper is to show some similarities and differences between
(\ref{model_intro}) and solutions of the classical heat equation. We
will show in particular that both models share positivity properties
such as the strong maximum principle. However solutions of
(\ref{model_intro}) do not smooth in time, except asymptotically as
$t\to \infty$.

The paper is organized as follows.  In Section
\ref{sec:metr-meas-spac} we present several  metric measure
spaces in which all the analysis carried out in this paper
holds. Those include open sets of the euclidean space, graphs, compact
manifolds, multi-structures (sets composed by several compact sets with
different dimensions joint together) or even some fractal sets. In Section
\ref{sec:line-nonl-diff} we  derive a comprehensive  study of the
linear operator $K_J-hI$. We will discuss in particular continuity and
compactness in different function spaces, including the case of
convolution-type operators.


We also study the positiveness of the diffusive operator $K_J$. Under
the assumption
 \begin{equation}\label{positivity_J_K_aum_pos}
		J(x,y)\!>\!0\mbox{ for all }x,\,y\in\Omega\mbox{, such that }d(x,y)
		\!<\!R,
\end{equation}
for some $R>0$ and the geometric condition that $\Omega$ is
$R$-connected (see Definition \ref {R_connected}), we show that for a
nonnegative nontrivial function $z$, the set of points in $\Omega$
where $K_Jz$ is strictly positive is larger than that of $z$.  This
will also allow us to use Kre\u{\i{}}n-Rutman Theorem, (see \cite
{krein_rutman}), to obtain that the spectral radius in
$\mathcal{C}_b(\Omega)$ of the operator $K_J$ is a positive simple
eigenvalue, with a strictly positive eigenfunction associated.   Condition
(\ref{positivity_J_K_aum_pos}) is also shown to be somehow optimal. 

In the last part of Section \ref{sec:line-nonl-diff} we study similar
questions for the  nonlocal operator 
$K_J-hI$, with $h\in L^{\infty}(\Omega)$. In particular, we derive
Green's formulas in the spirit of \cite{Rossi_libro} and characterize
the spectrum, which is also shown to be independent of the function
space.

In Section \ref{sec:line-evol-equat} we analyze the  solutions  of
\eqref{model_intro},  as   well as  the monotonicity properties of the
solutions. In particular we will show that
(\ref{positivity_J_K_aum_pos}) implies that \eqref{model_intro} has a
strong maximum principle. 

%

We then show that although solutions of (\ref{model_intro}) do not
regularize, because they carry the singularities of the initial data,
there is a subtle asymptotic smoothness for large times. In particular
 the semigroup $S(t)$ of \eqref{model_intro} is asymptotically smooth
 as in  \cite[p. 4]{Hale}. 

Finally, using the techniques of Riesz projections and the fact that
the spectrum is independent of the space, we are able to describe the
asymptotic behavior of the solutions of (\ref{model_intro}).

\section{Examples of metric measure spaces}
\label{sec:metr-meas-spac}

In the following sections we will consider a general measure metric
space $(\Omega,\mu,d)$ as in Definition \ref{def_measure}. Below we
enumerate some examples to which we can apply the theory developed
throughout this work.

$\bullet$ {\sc A subset of $\R^N$: } Let $ \Omega$ be a Lebesgue
measurable set of $\R^N$ with positive measure. A particular case is
the one in which $\Omega$ is an open subset of $\R^N$, which can be
even $\Omega=\R^N$. We consider the metric measure space
$(\Omega,\mu,d)$ where $\Omega\subseteq \R^N$, $\mu$ is the Lebesgue
measure on $\R^N$, and $d$ is the Euclidean metric of $\R^N$.

$\bullet$  {\sc Graphs: }  
We consider a non empty, connected and finite  graph in $\mathbb{R}^N$
defined by $G=(V,E)$, where $V\subset\mathbb{R}^N$ is the
finite set of {\it vertices}, and the {\it edge set} $E$, consists of
a collection of Jordan curves
\[
		E=\left\{\left.\pi_j: [0,\, 1]\to \mathbb{R}^N\right | \, j\in 
		\left\{1,\,2,\,3,...,\,n \right\}\right\}
\]
where $\pi_j\in\mathcal{C}^1\left( [0,1]\right)$ is injective with
$\pi_j(0), \pi_j(1) \in V$.  
We  identify the graph  with its associated network.
\[
G=\bigcup\limits_{j=1}^{n} e_j =\bigcup\limits_{j=1}^{n} \pi_j 
\left( [0,\, 1]\right) \subset \mathbb{R}^N 
\] 
and we  assume that any two edges $e_j\ne e_h$ satisfy that the intersection $e_j\cap e_h$
is either empty, one vertex or two vertices.

We define the measure structure of this graph. The edges have 
associated the one dimensional Lebesgue measure. Hence a set $A\subset
e_i$ is {\bf measurable} if and only if	 $\pi_i^{-1}(A)\subset[0,1]$
is measurable, and for any measurable set  $A 	\subset e_i$, we
consider the measure $\mu_i$, defined as   
\begin{displaymath}
\mu_i
(A)=\int_{\pi_i^{-1}(A)}\|\pi_i'(t)\|dt .   
\end{displaymath}
 In particular,  the length
of the edge $e_i$ is defined as the length of the curve  $\pi_i$,
	\begin{equation}\label{length_curve}
		\mu_{i}(e_i) = \int_0^1 \|\pi_i' (t)\| 
		dt .
	\end{equation}
Therefore,  a set $A\subset G$ is {\bf measurable} if and only if
$A\,\cap\, e_i$ is measurable for every $i\in
\left\{1,\,2,\,3,...,\,n\right\}$, and its measure  is given by 
	\begin{displaymath} 
		\mu_{G} (A)=\sum\limits_{i=1}^{n} \mu_i(A\cap e_i).
	\end{displaymath}
	
With this, it turns out that a function $f:G\to\mathbb{R}$
is measurable if and only if $f_{|e_i}:e_i\to\mathbb{R}$ is
measurable, if and only if $f \circ \pi_{j} : [0,1] \to \R$ is
measurable.

For $1\le p<\infty$, we set  $f\in L^p(G)=\prod_{i=1}^{n} L^p(e_i)$,
with norm $\|f\|_{L^p(G)}=\sum_{i=1}^{n} \|f\|_{L^p(e_i)}<\infty$, 
where, $\|f\|_{L^p(e_i)}=\left(\int_0^1 \left|f(\pi_i(t))\right|^p
  \|\pi_i' (t)\| dt\right)^{1/p}=\left(\int_0^1
  \left|f(\pi_i(\cdot))\right|^p d\mu_i \right)^{1/p}$. For $p=\infty$,
$f\in L^{\infty}(G)=\prod_{i=1}^{n} L^{\infty}(e_i)$, with norm
$\|f\|_{L^{\infty}(G)}=\max_{i=1,\dots,n} \|f\|_{L^{\infty}(e_i)}< \infty$,
where, $\|f\|_{L^{\infty}(e_i)}=\sup_{t\in [0,1]} \left|
  f(\pi_i(t))\right|$.

Now, we  describe the metric associated to the 
graph. For $v,\,w\in G$ the {\bf geodesic distance} from $v$ 
to $w$, $d_{g}(v,w)$,  is the length of the shortest path from $v$ to
$w$. This distance, $d_g$,  defines  the metric structure  associated
to the graph $G$. Observe that  since the graph is connected, 
there always exists the path from $v$ to $w$, and since the graph is 
finite the geodesic metric $d_g$ is equivalent to euclidean metric in
$\R^N$.  With this, a continuous  function $f:G\to\R$ has a norm
$\|f\|_{\mathcal{C}(G)}=\max_{i=1,\dots,n}
\|f\|_{\mathcal{C}(e_i)}<\infty$, where
$\|f\|_{\mathcal{C}(e_i)}=\sup_{t\in 
  [0,1]} \left|f(\pi_i(t))\right|$. 
	
Thus the graph defines a  metric measure space  $(G,\mu_G,d_g)$.

$\bullet$  {\sc Compact Manifolds: } Let $\mathcal{M}\subset\R^N$ be a 
compact manifold that we define  as follows.  Let  $U$ be an open 
bounded set  of $\R^d$, with $d\le N$, and let  $\varphi:U\to 
\R^N$ be an application such that it defines a diffeomorphism  from 
$\overline{U}$ onto its image $\varphi(\overline{U})$. Then we
define the compact manifold as
$\mathcal{M}=\varphi(\overline{U}) $.

A natural measure in $\mathcal{M}$ is the one for which, $A\subset 
\mathcal{M}$ is measurable if and only if $\varphi^{-1}(A)\subset\R^d$ 
is measurable. Hence for any measurable set  $A\subset \mathcal{M}
$, we define the measure $\mu$ as, 
\begin{equation}\label{measure_compact_variedad}
	\mu(A)=\int_{\varphi^{-1}(A)}\sqrt{g}\,dx, 
\end{equation} 
where $g=det(g_{ij})$ and  $g_{ij}\!=\!\langle{\partial \varphi\over \partial x_i},
{\partial \varphi\over \partial x_j}\rangle$.  Since the
compact manifold $\mathcal{M} =\varphi(\overline{U}) \subset
\R^N$ and $U \subset \R^{d}$,  then  the measure
\eqref{measure_compact_variedad} is equal to the  
$d$-Hausdorff measure in $\R^{N}$ restricted to $\mathcal{M}$, 
(see \cite[p. 48]{Simon}).

To define a  natural metric in $\mathcal{M}$, let   $\ell(c)$ be the
length of a  curve, $c$, in $\R^{N}$  defined as  	in
\eqref{length_curve}. Then we define  the {\bf geodesic distance} between two points
$p, q$ in the  manifold $\mathcal{M}$ as  
\begin{displaymath} 
d_g(p,q):=\inf \{\ell(c)\,|\; c:[0,1]\to \mathcal{M}\;\mbox{ smooth 
curve, } c(0)=p, \,c(1)=q \}.
\end{displaymath}
Since $\mathcal{M}\subset\R^N$ is compact, the geodesic metric, $d_g$, and the 
euclidean metric of $\R^{N}$, $d$, are equivalent. 

Thus we have  the metric measure space  $(\mathcal{M},\mathcal{H}^d,d)$ 
where $\mathcal{H}^d$ is the $d$-dimensional Hausdorff measure in
$\R^{N}$  and $d_g$ is the geodesic metric, which is equivalent to the
Euclidean metric  of $\R^N$.  
	
	$\bullet$  {\sc Multi-structures:} Now, we consider a multi-structure, 
	composed 
	by several  compact sets with different dimensions. 
	For example, we can think 
	in a piece of plane joined to a curve that is joined to a sphere in 
	$\R^N$, or we 
	can think also in a 
	dumbbell domain. Therefore, we  are going to define an appropriate 
	measure and metric for these multi-structures. 
	
Consider a collection of  metric measure spaces $\big\{\!(X_i,\mu_i, d_i)\!\big\}
_{i\in \{1,\dots,n\}}$, with its respective measures, $\mu_i$,   and 
metrics, $d_i$, defined as above. Moreover, we assume  
 the  measure  spaces  $\big\{(X_i,
 \mu_i)\big\}_{i\in \{1,\dots,n\}}$ satisfy  
$$
\mu_i(X_i\cap X_j)=\mu_j(X_i\cap X_j)=0,
$$ 
for $i\ne j$, and $i,\,j\in\{1,\dots,n\}$.

Then we define
	\begin{displaymath} 
		X=\bigcup\limits_{i\in \{1,\dots,n\}}X_i,
	\end{displaymath}
	and we say that $E\subset X$ is measurable if and only if $E\cap X_i$ 
	is $\mu_i$-measurable for all $i\in \{1,\dots,n\}$. Moreover we define the 
	measure 
	$\mu_X$ as
	\begin{displaymath} 
		\mu_X(E)=\sum\limits_{i=1}^n\mu_i(E\cap X_i).
	\end{displaymath}

Now  let us define the metric  that we 
	consider in $X$. We  assume that  
	$X_i\subset\R^N$ is compact for all $i\in\{1,\dots,n\}$, and the 
	metrics $d_i$ associated to each $X_i$, are  
	equivalent to the 	euclidean metric in $\R^N$. Therefore, the metric $d$ 
	that we consider 	for the multi-structure, is the euclidean 	metric in $
	\R^N$. 

Thus, we have the metric measure space, $(X,\mu_X,d)$, which is called
the direct sum of  metric measure spaces $(X_i,\mu_i, d_i)$, $i\in
\{1,\dots,n\}$.

$\bullet$  {\sc Spaces with finite Hausdorff measure and geodesic distance:}  
There exist examples of compact sets  $F\subset \R^N$ of Hausdorff
dimension $\mathcal {H}_{dim}(F)=s <N$ and finite $s$-Hausdorff
measure, i.e.,  $\mathcal{H}^s(F)<\infty$, which are pathwise
connected, with finite length paths. Some of this sets can be
constructed as self-similar affine fractal sets, and such an example
is provided by the Sierpinski gasket, see e.g. \cite{Kigami}, \cite{Lapidus} and
\cite{Lapidus_2}.  

For such sets, we can consider the metric measure space
$(F, \mathcal{H}^s ,d_{g})$ where  $d_g$ is the
geodesic metric  which may not be equivalent to the euclidean metric
in $\R^N$.   
\begin{figure} [htp]
\begin {center}
\includegraphics [height=4cm]{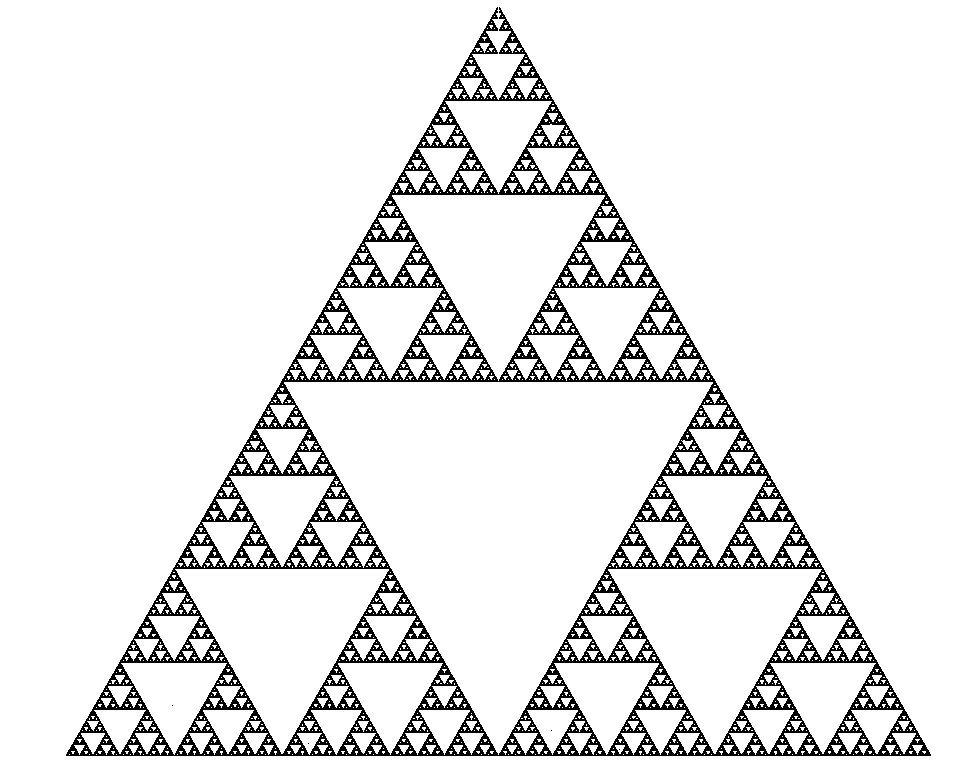}
\caption{Sierpinski Gasket.}\label{Sierpinski}
\end {center}
\end{figure}

\section{The linear nonlocal diffusion operator}
\label{sec:line-nonl-diff}

Let $(\Omega,\mu,d)$ be a metric measure space and consider a linear
nonlocal diffusion operator of the form 
$$ 
K_Ju(x)=\int_{\Omega}J(x,y)u(y)dy , 
$$  
where the  function $J$, defined in $\Omega$ as 
\[
\Omega\ni x\mapsto J(x,\cdot)\ge 0 . 
\]
We will not assume, unless otherwise made explicit, 
that $\Omega$ has a finite measure nor that $\int_{\Omega} J(x,y)\, dy
  =1$.



Hereafter  for $1\le p \le \infty$  we will denote by $p'$ its
conjugate exponent, that is, satisfying
$1=1/p+1/p'$. Notice that  the dual space of $L^p(\Omega)$ is given
by  $\left(L^p(\Omega)\right)'=L^{p'}(\Omega)$, for $1\le p< \infty$,  and for
$p=\infty$, $\left(L^{\infty}(\Omega)\right)'=\mathcal{M}(\Omega)$,
where $\mathcal{M}(\Omega)$~ is the space of Radon measures, for more
information see \cite [chap. 7]{Folland}.

\subsection{Properties of the operator $K_J$}

We begin with the following result. 

\begin{prop} \label{prop1}  \noindent
	\renewcommand{\labelenumi}{\roman{enumi}.}
	\begin{enumerate}
		\item Assume  $1\le p,\, q\le\infty$ and  $\;J\in L^q(\Omega,
		L^{p'}(\Omega))$. Then  $K_J\in \mathcal{L}(L^p(\Omega),	L^q(\Omega))$ 
		and the mapping $J\longmapsto K_J$ is linear and continuous,  
		and
		\begin{equation}\label{result_i}
			\|K_J\|_{\mathcal{L}(L^p(\Omega),L^q(\Omega))}\le
			\|J\|_{L^q(\Omega,L^{p'}(\Omega))}.
		\end{equation}

		\item Assume  $1\le p\le\infty$,  $J\in L^{\infty}(\Omega,L^{p'}
		(\Omega))$ and for any measurable set $D\subset \Omega$ satisfying $\mu(D)<\infty$,
		\begin{equation}\label{hyp_ii}
	 		\lim\limits_{x\rightarrow x_0}\displaystyle\int_{D}J(x,y) dy=
			\int_{D}J(x _0,y)dy,\qquad \forall x_0\in
                        \Omega. 
		\end{equation} 
		Then  $K_J\in \mathcal{L}(L^p(\Omega),
		\mathcal{C}_b(\Omega))$ and the mapping 
		$J\longmapsto K_J$	is linear and continuous, and
		\begin{equation}\label{result_iii}
		\|K_J\|_{\mathcal{L}(L^p(\Omega),\mathcal{C}_b(\Omega))}
		\le 
		\|J\|_{L^{\infty}(\Omega,\, L^{p'}(\Omega))}.
		\end{equation}
		In particular, if $J\in \mathcal{C}_b(\Omega,L^{p'}(\Omega))$, 
		then $K_J\in \mathcal{L}(L^p(\Omega),
		\mathcal{C}_b(\Omega))$, and
		$$\|K_J\|_{\mathcal{L}(L^p(\Omega),\mathcal{C}_b
		(\Omega))} \le \|J\|_{\mathcal{C}_b(\Omega,\, L^{p'}(\Omega))}.$$
		
		\item Assume $\Omega\subset\R^N$ is {\bf open},  $1\le p,\, q\le\infty$, 
		and $J\in W^{1,q}(\Omega,L^{p'}(\Omega))$. Then  $K_J\in \mathcal{L}(L^p(\Omega),W^{1,q}(\Omega))$ and the 
		mapping $J\longmapsto K_J$	is linear and continuous, and
		\begin{equation}\label{result_ii}
			\|K_J\|_{\mathcal{L}(L^p(\Omega),W^{1,q}(\Omega))}\le
			\|J\|_{W^{1,q}(\Omega,L^{p'}(\Omega))}.
		\end{equation}
	\end{enumerate}
\end{prop}
\begin{proof}
	\mbox{}

	\noindent {\it i.$\qquad$}  Thanks to H\"older's inequality, we have for 
	$1\le q<\infty$ and $1\le p\le \infty$,
	\[
	\begin{array}{lll}
	\|K_Ju \|^q_{L^q(\Omega)}
	&=\displaystyle{\int_{\Omega}\left| \int_{\Omega} J(x,y)u(y)\,dy \right|
	^qdx}
	\smallskip\\
	&\le\displaystyle \|u\|^q_{L^p(\Omega)}\displaystyle{\int_{\Omega}\|
	J(x,\cdot)\|^q_{L^{p'}(\Omega)}dx}
	&  =\|u\|^q_{L^p(\Omega)}\|J\|^q_{L^q\left(\Omega,L^{p'}(\Omega)
	\right)}.
	\end{array}
	\]
	\indent For $q=\infty$ and $1\le p\le \infty$, for each $x\in\Omega$,
	\begin{displaymath}
		|K_Ju (x)|  = \displaystyle{\left|\int_{\Omega}J(x,y)u(y)dy\right|} 
		\le \|u\|_{L^p(\Omega)}\|J(x,\cdot)\|_{L^{p'}(\Omega)}, 
	\end{displaymath}
	and taking supremum in $x \in \Omega$, we obtain the result.

	\noindent {\it ii.$\qquad$} Note that  since $J\in L^{\infty}(\Omega, L^{p'}
	 (\Omega))$, from part  {\it i.} with $q=\infty$, we have that $K_J\in
	 \mathcal{L}(L^p(\Omega), L^{\infty}(\Omega))$. Also note that the  
	 hypothesis \eqref{hyp_ii} can also be written as 
	 \begin{displaymath}  
	 	\lim\limits_{x\rightarrow x_0}\displaystyle\int_{\Omega}J(x,y)
		\chi_D(y)dy=
		\int_{\Omega}J(x _0,y)\chi_D(y)dy,\qquad \forall x_0\in \Omega,
	\end{displaymath}
	where $\chi_D$ is the characteristic function of $D\subset
	\Omega$, with $\mu(D)<\infty$, which means that $K_J(\chi_D)$
	is continuous and bounded in $\Omega$. Since $\mu(D)<\infty$,
	then $\chi_D\in L^p(\Omega)$, for $1\le p\le \infty$. 
	Moreover, the space 
	$$
        V=span\left[\,\chi_D\,;\, D\subset\Omega\;\;\mbox
	{with }\;\mu(D)<\infty\right],
        $$ 
	is dense in $L^p(\Omega)$, for   $1\le p \leq  \infty$ and $K_J:V
	\rightarrow \mathcal{C}_b (\Omega)$, and  then  
	$$
	K_J(L^p(\Omega))=K_J\left(\overline{V}\right)\subset \overline{K_J
	  (V)}\subset \mathcal{C}_b(\Omega)
	$$ 
	and we get (\ref{result_iii}). 
	
	 In particular if $J\in C_b\big(\Omega,L^{p'}(\Omega)\big)$, then the 
	hypothesis \eqref{hyp_ii} is satisfied.

	\noindent{\it iii.$\qquad$}  As a consequence of Fubini's
        Theorem, and since $\Omega$ is open we have that  for all  $u\in 
	L^p(\Omega)$ and $i=1,\dots,N$, the weak derivative of $K_Ju $ is 
	given   by
	\begin{equation}\label{derivative_of_K_J}
	\begin{array}{lll}
		\left\langle {\partial \over\partial x_i}K_{J}u ,\varphi \right\rangle & =-\left\langle 
		K_Ju ,\partial_{x_i} \varphi \right\rangle
		=-\displaystyle\int_{\Omega}\int_{\Omega}J(x,y)u(y)\partial_{x_i}
		\varphi(x)\,dy\,dx&
		\smallskip\\
		&=-\left\langle \left\langle J(\cdot,y),\partial_{x_i}\varphi \right\rangle,
		u \right\rangle
                = \left\langle \left\langle\partial_{x_i} J(\cdot,y),\varphi \right\rangle,
		u \right\rangle&
		\smallskip\\
		&=\displaystyle\int_{\Omega}\int_{\Omega}\partial_{x_i}J(x,y)u(y)
		\,\varphi
		(x)\,dy\,dx 
		 =\langle K_{\frac{\partial J}{\partial x_i}}u ,
		\varphi\rangle.	
	\end{array}
	\end{equation}
	for all  $\varphi\in \mathcal{C}^{\infty}_c(\Omega)$. Therefore 
	\begin{equation}\label{derivative_K_J_der_J}
		\,\frac{\partial}{\partial x_i}K_{J}u =K_{\frac{\partial J}{\partial x_i}}u .
	\end{equation}
 	Since $J\in W^{1,q}(\Omega,L^{p'}(\Omega))$, and from part {\it i.} and 
	\eqref{derivative_K_J_der_J}, we have that 
	\begin{equation}\label{derivative_of_K_J_2}
		\|K_J\|_{\mathcal{L}(L^p(\Omega),L^q(\Omega))}\le
			\|J\|_{L^q(\Omega,L^{p'}(\Omega))}
	\end{equation}
	and for $i=1,\dots,N$,
	\begin{equation}\label{derivative_of_K_J_3}
		\left\|\frac{\partial}{\partial x_i}K_{J}\right\|_{\mathcal{L}(L^p(\Omega),L^q(\Omega))}=
		\left\|K_{\frac{\partial J}{\partial x_i}}\right\|_{\mathcal{L}(L^p(\Omega),L^q(\Omega))}\le
		\left\|\frac{\partial J}{\partial x_i}\,\right\|_{L^q(\Omega,L^{p'}(\Omega))}.
	\end{equation}
	Hence, $K_J\in \mathcal{L}(L^p(\Omega),W^{1,q}(\Omega))$, for all $1\le p,q\le 
\infty$ and from \eqref{derivative_of_K_J_2} and \eqref{derivative_of_K_J_3} we have 
\eqref{result_ii}.
\end{proof}

The following result collects  cases in which $K_J\in\mathcal{L}(X,X)$, 
with $X=L^p(\Omega)$ or $X=\mathcal{C}_b(\Omega)$.
\begin{corolario}\noindent
	\begin{enumerate}
	\renewcommand{\labelenumi}{\roman{enumi}.}
	 \item If $J\in	L^p(\Omega, L^{p'}(\Omega))$ then  
	$K_J\in \mathcal{L}(L^p(\Omega), L^p(\Omega))$. 

	 \item If $J\in	\mathcal{C}_b(\Omega, L^1(\Omega))$ then  
	$K_J\in \mathcal{L}(\mathcal{C}_b(\Omega), \mathcal{C}_b(\Omega))
	$.
	 \item If $\mu(\Omega)<\infty$ and $J\in L^{\infty}(\Omega, L^{\infty}
	 (\Omega))$ then  $K_J\in \mathcal{L}(L^p(\Omega), L^p(\Omega))$, 
	 for all $1\le p\le \infty$.
	\end{enumerate}
\end{corolario}
\begin{proof} 	{\it i.}~ From Proposition \ref{prop1} we have
        the result. 
	
	{\it ii.}~If $J\in\mathcal{C}_b(\Omega, L^1(\Omega))$ then, thanks to the previous 
	Proposition \ref{prop1}, $K_J$ belongs to $\mathcal{L}(L^{\infty}(\Omega), \mathcal{C}_b
	(\Omega))$. Moreover, since $\mathcal{C}_b(\Omega)\subset 
	L^{\infty}(\Omega)$, we have that $K_J\in\mathcal{L}(\mathcal{C}_b
	(\Omega), \mathcal{C}_b(\Omega))$. 
	
	{\it iii.}~ From Proposition \ref{prop1} we have that $K_J\in\mathcal
	{L}(L^1(\Omega), L^{\infty}(\Omega))$. Moreover, since $\mu(\Omega)
	<\infty$, 
	\[
	L^p(\Omega)\hookrightarrow L^1(\Omega) \stackrel{K_J}{\longrightarrow}  
	L^{\infty}(\Omega)\hookrightarrow L^p(\Omega).
	\]
\end{proof}

The particular case where the  nonlocal diffusion term is given by a
convolution in  $\Omega=
\R^N$  with a function $J_0:\R^N\to\R$, i.e. $J(x,y)=J_0(x-y)$ and    
$K_{J}u =J_0\ast u$,  
has been widely considered, e.g. 
\cite{intro,ref1,Rossi,ref5} and references therein.  
Hence, we consider here such type of operators. For this, let
$\Omega\subset \R^N$ be a measurable set, (it can be  
$\Omega=\R^N$, or just a subset $\Omega\subset \R^N$) and consider 
the kernel   
\begin{equation}\label{J_def_por_J_0}
	J(x,y)=J_0(x-y), \qquad  x,\,y\in\Omega.
\end{equation} 
where $J_0$ is a function in $L^{p'}(\mathbb{R}^N)
$, for $1\le p \le\infty$,  
and the nonlocal  operator
\begin{displaymath} 
	K_{J}u (x)=\int_{\Omega}J_0(x-y)u(y)dy, \quad x\in \Omega.
\end{displaymath}

Straight from Proposition \ref{prop1} we get the following. 
\begin{corolario}
	For $1\le p\le\infty$, let $\Omega\subseteq\R^N$ be a measurable set, 
         $J_0\in L^{p'}(\mathbb{R}^N)$ and $J$ defined in
         (\ref{J_def_por_J_0}). Then  $K_{J}\in
	\mathcal{L}(L^p(\Omega),L^{\infty}(\Omega))$. In particular if $\mu
	(\Omega)<\infty$, then $K_{J}\in \mathcal{L}
	(L^p	(\Omega),L^{q}(\Omega))$, for $1\le q\le \infty$.
\end{corolario}
\begin{proof}
	If $J_0\in L^{p'}(\mathbb{R}^N)$ then  $J \in  L^{\infty} 
	(\Omega,L^{p'}(\Omega))$, since 
	\[\displaystyle\sup\limits_{x\in\Omega}\|J(x,\cdot)\|_{L^{p'}(\Omega)}
	=\sup\limits_{x\in\Omega}\|J_0(x-\cdot)\|_{L^{p'}(\Omega)}\le\|J_0\|_
	{L^{p'}(\mathbb{R}^N)}<\infty.\] 
	Thus, thanks to 
	Proposition \ref{prop1}, we have that $K_{J}\in \mathcal{L}(L^p
	(\Omega),L^{\infty}	(\Omega))$. In particular, if
        $\mu(\Omega)<\infty$ 
	then $K_{J}\in \mathcal{L}(L^p(\Omega),L^{q}(\Omega))$, for all 
	$1\le q\le \infty$. 
\end{proof}
On the other hand, if $\mu(\Omega)=\infty$, (as in the case of 
$\Omega=\R^N$), then  $K_{J}$ is not necessarily in 
$\mathcal{L}(L^p(\Omega),L^{q}(\Omega))$, for $q\ne \infty$. In 
the proposition below we prove the cases which {\bf can not} be obtained 
as a consequence of Proposition \ref{prop1}.
\begin{prop} \label{prop1_convolution} With the notations above, 
	 let $\Omega\subseteq\R^N$ be a measurable set 
	with $\mu(\Omega)=\infty$ and let  $1\le p\le\infty$. 
	\renewcommand{\labelenumi}{\roman{enumi}.}
	\begin{enumerate}
	\item If $\;J_0\in L^{r}(\mathbb{R}^N)\;$ and $\frac{1}{q}=\frac{1}{p}+\frac{1}
	{r}-1$ then $K_{J}\in\mathcal{L}(L^p(\Omega), L^q(\Omega))$, and 
$$
\|K_{J}\|_{\mathcal{L}(L^p(\Omega),L^q(\Omega))}\le \|J_0\|_{L^r
	(\mathbb{R}^N)}.
$$
	In particular, if $r=1$ we can take $p=q$.
	\item  If  $\,\Omega\subset\R^N$ is {\bf open}, $J_0\in W^{1,r}(\mathbb{R}^N)$ and 
	$\,\frac{1}{q}=\frac{1}{p}+\frac{1}{r}-1\,$ then $K_{J}\in\mathcal{L}(L^p
	(\Omega),  W^{1,q}(\Omega))$, and 
$$
\|K_{J}\|_{\mathcal{L}(L^p
	(\Omega),W^{1.q}(\Omega))}\le \|J_0\|_{W^{1,r}(\mathbb{R}^N)}.
$$
	\end{enumerate}
\end{prop}
\begin{proof}
	\noindent 	
	{\it i.$\quad$} 
	If $u$ is defined in $\Omega$, let us denoted by $\hat{u}$ the
        extension by zero  of $u$ to $\R^N$. 
	Thus, we have for $x\in\Omega$
	\[K_{J}u (x)=\displaystyle\int_{\Omega}J_0(x-y)u(y)dy=\displaystyle\int_
	{\mathbb{R}^N}J_0(x-y)\hat{u}(y)dy=\left(J_0\ast\hat{u}\right)(x).\]
	Now, we define the extension of the operator $K_{J}$ as 
	$$
\widehat{K}_{J}u (x)=\left(J_0\ast\hat{u}\right)(x),\;\mbox{ for }\;x\in
	\mathbb{R}^N,
$$ 
then  
	$K_{J}u (x)=\left(\left.\widehat{K}_{J}u \right)\right|_{\Omega}\!\!(x)$, for 
	$x\in\Omega$. Thanks to Young's inequality, see
        \cite[p. 104]{Brezis},  we have
        \begin{displaymath}
          	\|K_{J}u \|_{L^q(\Omega)}\le \|\widehat{K}_{J}u \|_{L^q
	(\mathbb{R}
	^N)} \le\|J_0\|_{L^r(\mathbb{R}^N)}\|\hat{u}\|_{L^p(\mathbb{R}^N)}=
	\|J_0\|_
	{L^r(\mathbb{R}^N)}\|u\|_{L^p(\Omega)}.
        \end{displaymath} 
	Hence, $\|K_{J}u \|_{L^q(\Omega)}\le\|J_0\|_{L^r(\mathbb{R}^N)}\|u\|_
	{L^p	(\Omega)}$, for all $p,\,q,\,r$ such that  $\frac{1}{q}=\frac{1}{p}+\frac{1}
	{r}-1$.\\

	\noindent {\it ii.$\quad$} Following the same arguments made in Proposition 
	\ref{prop1} in \eqref{derivative_of_K_J}, we know that for $x\in \Omega$, 
	$$\;\frac
	{\partial}{\partial_{x_i}}K_{J}u =K_{\frac{\partial {J}}{\partial x_i}}u =
	\left. \left	(\widehat{K}_{\frac{\partial {J}}{\partial x_i}}u \right)\right|_
	{\Omega}$$ 
	Then, applying part {\it i.} to $\|K_{J}u \|_{L^q(\Omega)}$ and 
	$\|K_{\partial {J}\over \partial x_i}u \|_{L^q(\Omega)}$ we have that 
	 for $p,q,r$ such that 
	$\frac{1}{q}=\frac{1}{p}+\frac{1}{r}-1$, $K_{J}\in\mathcal{L}(L^p(\Omega), 
	W^{1,q}(\Omega))$. Thus, the result.
\end{proof}

Now we prove that under  the hypotheses on $J$ in Proposition
\ref{prop1},  the operator $K_J$ is compact. 
For this we will use the following result. 
\begin{lema}\label{finite_rank}
	For $1\le q<\infty\;$ and $\;1\le p\le\infty$, let $(\Omega,\mu)$ be a 
	measure space, then any function $H\in L^q(\Omega, L^{p'}(\Omega))$ 
	can be approximated in $L^{q}(\Omega,
	L^{p'}(\Omega))$ by functions of separated variables.
\end{lema}

Then we have. 
\begin{prop}
\label{K_compact}  \noindent
	\renewcommand{\labelenumi}{\roman{enumi}.}
	\begin{enumerate}
		\item For $1\le p \le\infty$ and $1\le q<\infty$, if $J\in L^q(\Omega, L^{p'}
	(\Omega))\,$ 
	then  $\,K_J\in\mathcal{L}(L^p(\Omega),L^q(\Omega))$ is compact. 
		\item For $1\le p\le \infty$, if $\,J\in BUC(\Omega,
                  L^{p'}(\Omega))$, then $K_J\in
                  \mathcal{L}(L^p(\Omega), 
	\mathcal{C}_b(\Omega))$ is compact. In particular,  $K_J\in \mathcal
	{L}(L^p(\Omega),L^{\infty}(\Omega))$ is compact.
		\item For $1\le p \le\infty$ and $1\le q<\infty$, if $\,\Omega\subset\R^N$
	 is {\bf open} and $J\in W^{1,q}(\Omega,L^{p'}(\Omega))\,$ then 
	$\,K_J\in\mathcal{L}(L^p(\Omega),W^{1,q}(\Omega))$ is compact. 
	\end{enumerate}
\end{prop}
\begin{proof} 

{\it i.}  Since $J\in L^q(\Omega, L^{p'}(\Omega))$, for $1\le p \le\infty$ 
	and 	$1\le q<\infty$, we know from Lemma \ref{finite_rank} that there exist 
	$\,M(n)\in\mathbb{N}\,$ and $f_j^n\in L^q(\Omega),\, g_j^n\in L^{p'}
	(\Omega)$ with $\,j=1,..., M(n)\,$ such that $J(x,y)$ can be approximated 
	by functions that are a finite linear combination of functions with 
	separated variables defined 
	as, 
$J^n(x,y)=\sum_{j=1}^{M(n)}f_j^n(x)g_j^n(y)$ 
and $\|J-J^n\|_{L^q(\Omega,L^{p'}(\Omega))}\rightarrow 0$, as $n$ goes 
to $\infty$.	
Then  define $K_{J^n}u (x)= {\sum_{j=1}^{M(n)}f_j^n(x)\int_
	{\Omega}g_j^n(y)u(y)dy}$.  
Thus,  since $K_J-K_{J^n}=K_{J-J^n}$, thanks to Proposition \ref{prop1},  we 
	have that,
$$
\|K_J-K_{J^n}\|_{\mathcal{L}(L^p(\Omega),L^q(\Omega))}\le \|J-J^n\|_{L^q
	(\Omega,L^{p'}(\Omega))} \rightarrow
        0, \quad \mbox{as $n$ goes to $\infty$} . 
$$
Since $K_{J^n}$ has  finite rank, then $K_J\in\mathcal{L}\left(L^p
	(\Omega),L^q(\Omega)\right)$ is compact, e.g.   
	\cite[p.157]{Brezis}. 

	\indent{\it ii.} 
	If $J\in BUC(\Omega, L^{p'}(\Omega))$, then hypothesis \eqref{hyp_ii} 
	of Proposition \ref{prop1} is satisfied and then $K_J\in\mathcal{L}
	(L^p(\Omega), \mathcal{C}_b(\Omega))$. Now, we consider $u\in B
	\subset L^p(\Omega)$, where $B$ is the unit 
	ball in $L^p(\Omega)$. Now, we prove using Ascoli-Arzela 
	Theorem (see \cite[p. 111]{Brezis}), that $K_J(B)$ is relatively 
	compact in $\mathcal{C}_b(\Omega)$.	Let  $x,z\in\Omega$,
        $\;u\in B$, thanks to H\"older's 
        inequality, we have,
        \begin{displaymath}  
          	|K_Ju (z)-K_Ju (x)|  =\left| \int_{\Omega} \Big(
                  J(z,y) - J(x,y) \Big) u(y)dy\right|
      \le\|J(z,\cdot)-J(x, \cdot)\|_{L^{p'}(\Omega)}. 
        \end{displaymath}
	Since $J\in BUC(\Omega, L^{p'}(\Omega))$, then for all $\varepsilon>0$, 
	there exists $\delta>0$ such that if $x,\,z\in
	\Omega$ satisfy that $d(z,x)<\delta$, then $ \|J(z,\cdot)-J(x,\cdot)\|_{L^{p'}
	(\Omega)}<\varepsilon$. Hence, we have that $K_J(B)$ is equicontinuous.
	
	 On the other hand, thanks to H\"older's inequality, for all $
         x\in \Omega$ and $u\in B$
	 \[
	 \begin{array}{ll}
	 	|K_Ju (x)| &=\displaystyle\left| \int_{\Omega}J(x,y)u(y)dy\right|
		\le \|J(x,\cdot)\|_{L^{p'}(\Omega)}<\infty.
	 \end{array}
	 \]
	Thus,    we 
	 have that $K_J(B)$ is precompact and therefore  
	 $K_J\in\mathcal{L}(L^p(\Omega),\mathcal{C}_b(\Omega))$ is 
	 compact. Also, the  second part of the result is immediate. 
	
	\indent{\it iii.} Thanks to the argument \eqref{derivative_of_K_J} in 
	Proposition  \ref{prop1}, we have that $\,\frac{\partial}{\partial x_i}K_{J}
	u =K_{\frac {\partial J}{\partial x_i}}u $. Since $J\in W^{1,q}
	(\Omega,L^{p'}(\Omega))$,  we have that $J\in L^{q}(\Omega,L^{p'}
	(\Omega))$ and moreover ${\partial J\over \partial x_i}\in L^{q}
	(\Omega,L^{p'}(\Omega))$, for all $i=1,\dots,N$. 
	Using part {\it i.} we obtain that $K_{{\partial J	\over\partial x_i}}\in
	\mathcal
	{L}\left(L^p(\Omega),\, L^q(\Omega)\right)$ is compact. 
	Thus, if $B$ is the unit ball in $L^p(\Omega)$, we have that $K_J(B)$ 
	and 	$K_	{{\partial J\over\partial x_i}}(B)$ are precompact for all 
	$i=1,\dots,N$. 	
	 Now we consider the mapping 
	$$
	\begin{array}{rccc}
	\mathcal{T}:&L^p(\Omega) &\longrightarrow &
	\!\!\!\!\!\!\!\!\!\!\!\!\!\!\!\!\!\!\!\!\!\left
	( L^q(\Omega)\right)^{N+1}\\
	&u&\longmapsto & \left(K_Ju ,\,K_{{\partial J
	 \over\partial x_1}}u ,\dots,\, 
	K_{{\partial J\over\partial x_N}}u \right).
	\end{array}
	$$ 
	Thanks to Tíkhonov's Theorem, we know 
	that 	$\mathcal{T}(B)$ 	is precompact in $\left(L^q(\Omega)\right)^{N
	+1}$. 
	Moreover, we consider the mapping
	$$
	\begin{array}{rccc}
	\mathcal{S}:&W^{1,q}(\Omega) &\hookrightarrow &\left( L^q(\Omega)
	\right)^{N+1}\\
	&g&\longmapsto & \left(g,{\partial g\over\partial x_1},\dots, {\partial g
	\over
	\partial x_N}
	\right).
	\end{array}
	$$ 
	 Since $\mathcal{S}$ is an isometry, 
       then we have 	 that 
	 $\mathcal{S}^{-1}|_{Im(\mathcal{S})}:  Im(\mathcal{S})\subset \left( L^q
	 (\Omega) \right)^{N+1}\rightarrow W^{1,q}(\Omega)$ is continuous. 
	 On the other hand, thanks 
	 to the hypotheses on $J$ and Proposition \ref{prop1}, we have that 
	 $K_J\in
	 \mathcal{L}\left(L^p(\Omega),
	  \,W^{1,q}(\Omega)\right)$. Thus, $Im(\mathcal{T})\subset Im(\mathcal
	  {S})$.\\
	 Hence,  the operator $ K_J:L^p(\Omega)\to W^{1,q}(\Omega)$, can be 
	 written as
	  $$K_Ju =\mathcal{S}^{-1}|_{Im(\mathcal{S})} \circ\mathcal{T}u .$$
	 Therefore, we have that $K_J$ is the composition of a continuous 
	 operator $\mathcal{S}^{-1}|_{Im(\mathcal{S})}$, 
	 with  a compact operator $\mathcal{T}$. Thus, the result. 	
\end{proof}

\begin{remark}\label{same_hyp_prop1_for_compactness}
Observe that the assumptions in Proposition \ref{K_compact}  are the
same as in Proposition \ref{prop1} except for the case  
	$K\in \mathcal{L}(L^p(\Omega), L^{\infty}(\Omega))$ where  we 
	assume in the former that $J\in BUC(\Omega, L^{p'}(\Omega))$, instead of 
	$J\in L^{\infty}(\Omega, L^{p'}(\Omega))$ as in the latter. 

\end{remark}

Now we derive several consequences from interpolation. Note that the
following result is valid for a general operator $K$, not necessarily
an integral operator.
\begin{prop}\label{compacidad_cobos}
	Let $(\Omega,\mu)$ be a measure space, with $\mu(\Omega)<\infty$. 
	Assume that for  $1\le p_0< p_1<\infty$,  $K\in\mathcal{L}(L^{p_0}(\Omega),L^{p_0}
	(\Omega))$  and
        $K\in\mathcal{L}(L^{p_1}(\Omega),L^{p_1}(\Omega))$. Then
        $K\in\mathcal{L}(L^{p}(\Omega),L^{p}(\Omega))$, for all $p\in
        [p_0,\,p_1]$. Additionally, suppose that either: 
	\begin{enumerate}
		\renewcommand{\labelenumi}{\roman{enumi}.}
		\item
                  $K\in\mathcal{L}(L^{p_0}(\Omega),L^{p_0}(\Omega))$
                  is compact, or 
		\item
                  $K\in\mathcal{L}(L^{p_1}(\Omega),L^{p_1}(\Omega))$
                  is compact . 
	\end{enumerate}
	Then $K\in\mathcal{L}(L^p(\Omega),L^p(\Omega))$ is compact for all 
	$p\in[p_0,p_1]$.
\end{prop}
\begin{proof}
	From Riesz-Thorin Theorem,  (see 	\cite[p. 196]{Riesz_thorin}), 
	we have   
	$K\in\mathcal{L}(L^{p}(\Omega),L^{p}(\Omega))$, for all $p\in 
	[p_0,\,p_1]$. 
	The proof of the compactness can be found  in \cite[p. 4]{Cobos}.
\end{proof}

Now we analyze positive preserving properties of nonlocal
operators. For this we will  need some positivity properties of the
kernel $J$ and some connectedness of $\Omega$. 
To do this, we  first introduce the  following.

\begin{definition}\label{R_connected}
Let $(\Omega,\mu,d)$ be a metric measure space and $R>0$. We 
say that $\Omega$ is 
{\bf $R$-connected} if for any  
$x,y\in\Omega$, there exists a finite $R$-chain connecting $x$ and
$y$. By this we mean that there exist  $N\in \mathbb{N}$
and a finite set of points $\{x_0,\dots,x_{N}\}$ in $\Omega$ such that 
$x_0=x$, $x_N=y$ and 
$d(x_{i-1},x_{i})<R$, for all $i=1,\dots,N$.
\end{definition}
%

Then we have the following.  
\begin{lema}\label{compact_connected_R_connected}
	If $\Omega$ is compact and connected then $\Omega$ is 
	$R$-connected for any $R>0$. 
\end{lema}
\begin{proof} 	We fix an arbitrary $x_0\in\Omega$,  and we define the
        increasing sequence of open sets  
	\begin{equation}  \label{def_P_x_0}
	 P^1_{R,x_0}=B(x_0,R)\;\mbox{ and }\; 
	P_{R,x_0}^n=\!\!\!\!\bigcup\limits_{x\in P_{R,x_0}^{n-1}}\!\!\!\!\!\!B
	(x,R)\;	\mbox{ for }\;n\in\mathbb{N}.
	\end{equation} 
Observe that $P_{R,x_0}^n$ is the set of points in $\Omega$ that for
which there exists an $R$-chain of $n$ steps, connecting with
$x_{0}$. 
Then $A= \bigcup_{n=1}^{\infty} P_{R,x_0}^n$ is open. Lets us show
that it is also closed. In such a case since $\Omega$ is connected we
would have $\Omega =A$ which implies that $\Omega$ is $R$--connected,
since $x_{0}$ is arbitrary. Indeed if $y \in \Omega\setminus A$, then
we claim the $B(y,R) \subset \Omega \setminus A$, since otherwise
$B(y,R)$ would intersect some $P_{R,x_0}^n$, which implies that $y\in
P_{R,x_0}^{n+1}$ which is absurd. 
\end{proof}

With this, we get the following. 
\begin{lema}\label{R_conected_compact_set}
	Let $(\Omega,\mu,d)$ be a metric measure space such that $\Omega$ 
	is $R$-connected. For any fixed $x_0\in\Omega$ consider the
        sets   $P_{R,x_0}^n$ as in (\ref{def_P_x_0}). 
	 
	Then, for every compact set in $\mathcal{K}\subset 
	\Omega$, there exists $n(x_0)\in\mathbb{N}$ such that $\mathcal{K}
	\subset 	P_{R,x_0}^n$ 	for all $n\ge n(x_0)$. 
	
	Furthermore, if $\Omega$ is compact, there 
	exists $n_0\in\mathbb{N}$ such that for any $y\in \Omega$,
        $\Omega= P_{R,y}^{n}$, for all $n\ge n_0$. 
\end{lema}
\begin{proof}
	Since $\Omega$ is $R$-connected, for any $y\in \Omega$,
        consider an $R$-chain connecting $x_{0}$ and $y$,  $\{x_0,\dots,x_{M}\}$ such that  
	$x_{M}=y$ and $d(x_{i-1},x_{i})<R$, for all $i=1,\dots,M$. 
	Thus, $x_1\in B(x_0, R)=P^1_{R,x_0}$, $x_2\in B(x_1,R)\subset 
	P^2_{R, x_0}$,  $B(x_{i}, R)\subset P_{R,x_0}^{i+1}$, for 
	all $i=1,\dots, M$. In particular, $y\in P^M_{R,x_0}$ and  
	\begin{equation}\label{B_y_R_P_R_x_0}
		B(y,R)\subset P_{R,x_0}^{M+1}.
	\end{equation} 
	
	On the other hand, since $\mathcal{K}$ is compact, 
	 $\mathcal{K}\subset \bigcup_{y\in\mathcal{K}}B(y,R)$, 
	there exists $n\in\mathbb{N}$ and $y_{i} \in K$,  such that $\mathcal{K}
	\subset \bigcup
	_{i=1}^nB(y_i,R)$. From \eqref{B_y_R_P_R_x_0}, for every $i$ 
	there exists $M_{i}$ such 
	 that $B(y_i, R)\subset P_{R,x_0}^{M_{i}+1}$. 
	 We choose $n(x_0)=\max\limits_{i=1,\dots, n} \{M_{i}+1\}$, and we 
	 obtain  that  $\mathcal{K}\subset P_{R,x_0}^{n(x_0)}$. Therefore, 
	 $\mathcal{K} \subset P_{R,x_0}^{n}$, for all $n\ge n(x_0)$. Thus, the 
	 result. 
	 
	 If $\Omega$ is compact, from the 
	 previous result we know that for a  fixed $x_0\in\Omega$,
         there exists  $N=N
	 (x_0)$ such that $\Omega=P^{N}_
	 {R,x_0}$. Therefore, any two points in $\Omega$ are
         connected by an $R$-chain of $N$ steps to $x_{0}$. Thus any
         two points in $\Omega$ are connected to each other by and
         $R$-chain of $2N$ steps. In other words  $\Omega=
         P_{R,y}^{n}$, for all $n\ge 2N$ for all $y\in\Omega$.          
\end{proof}

  Now we define the {\it
  essential  support} of  a  nonnegative measurable function. 
 \begin{definition}\label{definition_positive_set}
	Let $z$ be a 
	measurable nonnegative function $z:\Omega\to\R $. We define the  
	{\bf essential support} of  $z$ as: 
	\begin{displaymath} 
		P(z)=\big\{x\in\Omega\, : \;\forall \delta>0,\; \mu\big(
		\{y\in\Omega:\;z(y)>0\}\cap B(x,\delta)\big)>0\big\},
	\end{displaymath}
 	where $B(x, \delta)$ is the ball 
	centered in $x$, with radius $\delta$.
\end{definition}
It is not difficult to check that  $z\ge 0$ not identically zero iff
$P(z)\ne\emptyset$ which is equivalent to  $\mu\left( P(z)
\right)>0$.

Let us introduce the following definitions.
\begin{definition}\label{sets_P_n}  
Let $z$ be a measurable nonnegative function $z:\Omega\to\R$. 
Then  we denote  
$$
P^0(z)=P(z),
$$
the essential support of $z$, and for any  $R>0$, we define the
increasing sequence of  open sets
$$
P_R^1(z)=\!\!\!\!\!\!\bigcup\limits_{x\in P^0(z)}\!\!\!\!\!\!B(x,R),
\qquad  P_R^2(z)=\!\!\!\!\!\!\bigcup\limits_{x\in
  P_R^1(z)}\!\!\!\!\!\!B(x,R), \quad \dots  \qquad  P_R^n(z)=\!\!\!\!\!\!\bigcup\limits_{x\in
  P_R^{n-1}(z)}\!\!\!\!\!\!B(x,R),
$$
 for all $n\in\mathbb{N}$.
 \end{definition}


Now, we prove the main result.
	
\begin{prop}\label{K_aumenta_positividad}
 	Let $(\Omega,\mu,d)$ be a metric measure space, and let $J\geq
        0$ satisfy 
	that 
	\begin{equation}\label{positivity_J_K_aum_pos_01}
		J(x,y)\!>\!0\mbox{ for all }x,\,y\in\Omega\mbox{, such that }d(x,y)
		\!<\!R,
	\end{equation}
	for some $R>0$.
	If $z\geq 0$ is a nontrivial measurable function defined in $\Omega$ then,
	 $$P(K_J^n(z))\supset P_R^n(z),\;\mbox{ for all }\;n\in\mathbb{N}.$$
	In particular, if $\Omega$ is $R$-connected, then for any 
	compact set $\mathcal{K}\subset\Omega$, 
	$$\exists\, n_0(z)\in\mathbb{N}, \;\mbox{ such that }\; P(K_J^n(z))\supset  
	\mathcal{K},\; \mbox{for all }\; n\ge n_0(z).$$
	If $\Omega$ is compact and connected, then $\exists\, n_0\in\mathbb
	{N}$,  such that,   for all $z\ge 0$ 
	measurable and  not identically zero
	$$ P(K_J^n(z))=
	\Omega,\; \mbox{for all }\; n\ge n_0.$$
\end{prop}
 \begin{proof}
 	First of all we prove that $P\left(K_Jz\right)\supset P_R^1(z)$. Since 
	$z\ge 0$, not identically zero, then  $\mu
        \left(P(z)\right)>0$ and then 
 	\begin{equation*}
		K_Jz(x)=\displaystyle\int_{\Omega}J(x,y)z(y)dy\ge \int_{P(z)}J
		(x,y)z(y)dy.
	\end{equation*}
	 From   \eqref{positivity_J_K_aum_pos_01}   we have  that 
	 \begin{equation}\label{pos_001}
	 K_Jz(x)>0\;\;\mbox{for all}\;\displaystyle x\in \!\!\!\bigcup
	  \limits_{y\in P(z)}\!\!\!\!B(y,R) =P^1_R(z).
	  \end{equation}
	 Since $P_R^1(z)$ is an open set in $\Omega$, 
	 we have that,  if $x\in P_R^1(z)$, then  
	 \begin{equation}\label{pos_002}
	 \mu\left(B(x,\delta)\cap P_R^1(z)\right)>0\quad  \mbox{for all}\; 
	 0<\delta . 
	 \end{equation}
	 Thus, thanks to \eqref{pos_001} 
	 and \eqref{pos_002}, we have that 
	 \begin{equation}\label{P_K_contains_P_1}
		P\left(K_Jz\right)\supset P_R^1(z).
	\end{equation}
	 Applying $K_J$ to $K_Jz$, following the previous arguments and 
	 thanks to \eqref{P_K_contains_P_1}, 
	 we obtain 
	 $$\displaystyle P\left(K_J^2(z)\right)\!\!\supset\!  P_R^1(K_Jz)\!=\!\!\!
	 \!\!\!\!\bigcup
	 \limits_{x\in P\left(K_Jz\right)}\!\!\!\!\!\!\!\!B(x,R)\supset \!\!\!\!\!\!\bigcup
	 \limits_{x\in P_R^1(z)}\!\!\!\!\!\!B(x,R)=P_R^2(z).$$
	 Therefore, iterating this process, we finally obtain that 
	 \begin{equation}\label{inc_0_1_2}
	 P(K_J^n(z))\supset P_R^n(z),\;\forall n\in\mathbb{N}.
	 \end{equation} 
	 
	 Now consider $\mathcal{K}\subset\Omega$ a compact set in $
	 \Omega$, and taking $x_0\in P(z)$, 
	 then thanks to Lemma \ref{R_conected_compact_set} there 
	 exists $n_0(z)\in\mathbb{N}$, such that $\mathcal{K}\subset\!P_R^{n}
	 (z)$ for all 	 $n\ge n_0(z)$,  then thanks to \eqref{inc_0_1_2}, 
	  $\mathcal{K}\subset P(K_J^n(z))$ for  all $n\ge n_0(z)$. 
	 
	 Now, if $\Omega$ is compact and connected, thanks to Lemma 
	 \ref{compact_connected_R_connected}, $\Omega$ is $R$-
	 connected. From Lemma \ref{R_conected_compact_set} there  exists 
	 $n_0\in\mathbb{N}$ such that  for any $y\in
	 \Omega$, $\Omega=P_
	 {R,y}^n$, for all  $n\ge n_0$. Hence, from \eqref{inc_0_1_2}, for any 
	 $z\ge 0$ not identically  zero, taking $y\in P(z)$,  $P(K_J^n(z))\supset 
	 P^n_{R,y}=\Omega,\; \forall 
	 \; n\ge n_0$.
 \end{proof}

%
%

 \begin{remark}
	Notice that  the hypothesis \eqref
	{positivity_J_K_aum_pos_01} 
	is somehow an optimal condition, as the following
        counterexample shows. 
	
	Let  $\Omega=[0,1] \subset\mathbb{R}$ and take
        $x_0=1/2$, and $0<R < 1/2$  	such that $(1/2-R,1/2+R)\subset [0,1]$. 
	We consider a function $J$ satisfying that $J\!\ge\! 0$  
	defined as 
	 \begin{equation}\label{pos_J_K_counterexample}
	 J(x,y)=\left\{\begin{array}{ll}
	\!\!1,&\!\!\!\!  \quad (x,\,y)\!\in\![0,1]^2\!\setminus \!
	 \big({1\over 2}\!-\!R,{1\over 2}\!+\!R\big)^2\mbox{, 
	 with }d(x,y)\!<\!R, \\
	 \!\!0 &\!\!\!\!\mbox{ for the rest of } (x,\, y).
	 \end{array}
	 \right.
	 \end{equation} 
%
	Now, we consider a function $z_0:\Omega\to\mathbb{R}\,$, $z_0\ge 
	0$,  	such that $P(z_0)\subset [1/2, 1]$. 
	 Since $z_0(y)=0$ in  $[0, 1/2]$, we have that  $K_Jz_0(x)=\!\displaystyle\int_{\Omega}\!\!J(x,y)z_0(y)dy\!= \! \int_{1/2}^{1}\!\!\!J(x,y)  
	z_0(y)dy$. 
	Moreover, from  \eqref{pos_J_K_counterexample}, we have 
	that for $\tilde{x}\in [0,1/2)$, $J(\tilde{x},y)=0$ for all
        $y\in [1/2, 1]$, (see Figure \ref{cuadrado}).
        \begin{figure} [htp]
	\begin {center}
	\includegraphics [height=5cm]{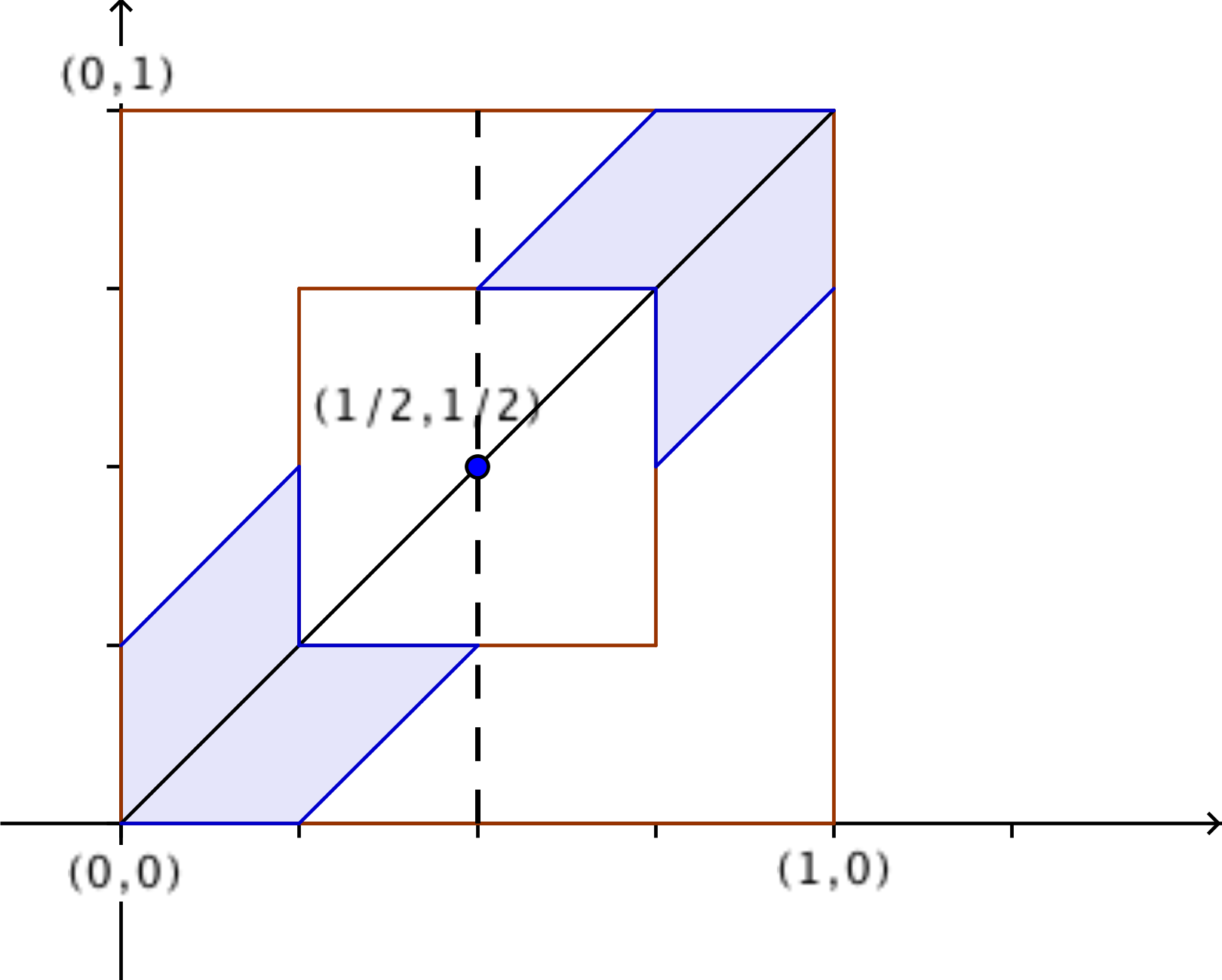}
	\caption{The shadowed area are the points $(x,y)$ where $J$ is
          strictly positive,  $R=1/4$. }
	\label{cuadrado}
	\end {center}
	\end{figure}
        
%
	Hence $K_Jz_0(\tilde x) =0$ in $ [0, 1/2)$, and therefore $P(K_Jz_0)
	\subset [1/2,1]$. 	
	Iterating this argument, we obtain that 
	$$
	P(K_J^n(z_0))\subset [1/2, 1]\quad\mbox{ for all }\;n\in\mathbb{N}.
	$$ 
\end{remark}
%



Now we describe the adjoint operator associated to $K_J$, and we prove
that if $J\in L^2(\Omega\times\Omega)$ and $J(x,y)=J(y,x)$ then the
operator $K_J$ is selfadjoint in $L^2(\Omega)$.

\begin{prop}\label{adjoint_operator}
	For $1\le\! p\!<\! \infty, \; 1\!\le \!q\!<\! \infty$. Let $(\Omega,\mu)$ be a 
	measure space. We assume that the mapping 
	$$x\mapsto J(x,\cdot)\;\mbox{ satisfies that }\; J\in L^q(\Omega,L^{p'}(\Omega)),
	$$
	 and  the mapping 
	 $$y\mapsto J^{*}(y,\cdot) := J(\cdot,y)\;\mbox{ satisfies that }\;J^{*}\in L^{p'}(\Omega,L^{q}
	(\Omega)).$$
	 Then the adjoint operator associated to $\,K_J\in\mathcal{L}(L^p(\Omega),
	  L^q(\Omega))$, is 
	 $$\,K_J^*:L^{q'}(\Omega)\rightarrow L^{p'}(\Omega),\; \mbox{ with} 
	 \;K_J^*=K_{J^*} . 
         $$

In particular, if  $J$ satisfies that
	 \begin{displaymath} 
	 	J(x,y)=J(y,x),
	 \end{displaymath} 
and $J\in L^2(\Omega\times \Omega)$, the operator $K_J$ is selfadjoint
in  $L^2(\Omega)$.
 \end{prop}
 \begin{proof} 
 	We consider $\, u\in L^p(\Omega)\,$ and $v\in L^{q'}(\Omega)$. Thanks 
	to Fubini's Theorem and the hypotheses on $J$
	\[
	\begin{array}{rl}
	\!\!\langle K_Ju ,v\rangle_{L^q(\Omega),L^{q'}(\Omega)} & =\displaystyle
	{\int_{\Omega}\int_{\Omega}J(x,y)u(y)dy\,v(x)dx} 
	\smallskip  = 
		\displaystyle{\int_{\Omega}\int_{\Omega}J(x,y)v(x)dx\,u(y)dy},
	\end{array}
	\]
	and $\displaystyle{\int_{\Omega}\int_{\Omega}J(x,y)v(x)dx\,u(y)dy}=
	\langle u,\, K_J^*v \rangle_{L^p(\Omega), L^{p'}(\Omega)},\;$with 
	\[\;K_J^*v (y)=\displaystyle{\int_{\Omega}J(x,y)v(x)dx}=\int_{\Omega}
	J^*(y,x)v(x)dx=K_{J^*}v (y). 
        \]


The symmetric case in $L^{2}(\Omega)$ is now obvious. 
\end{proof}

 

We will now  prove that under certain hypotheses on 
$K_J$ the spectrum  $\sigma_{X}(K_J)$ is 
independent of $X$, with $X=L^p(\Omega)$,  $1\le p\le\infty$ or 
$X=\mathcal{C}_b(\Omega)$. We also characterize the spectrum of $K_J$ 
when $K_J$ is selfadjoint in $L^2(\Omega)$, and prove   
that under the same hypothesis on the positivity of $J$ 
in Proposition  \ref{K_aumenta_positividad}, the spectral radius of 
$K_J$ in $\mathcal{C}_b(\Omega)$ is a simple eigenvalue  
that has a strictly positive associated eigenfunction. 

The proposition below is for a general compact operator 
$K$, not only for the integral operator $K_J$ (see Propositions \ref
{K_compact} to check compactness for operators with kernel, $K_J$). 

\begin{prop}\label{sigma_independiente}  
	Let $(\Omega,\mu,d)$ be a metric measure space with   $\mu(\Omega)<\infty$.
	\begin{enumerate}
	\renewcommand{\labelenumi}{\roman{enumi}.}
	\item Assume  $K\!\in\!\mathcal{L}(L^{p_0}(\Omega),L^{p_1}
	(\Omega))$  for some $1\!\le\! p_0\!<\! p_1\!<\!\infty$ and 
	$K\!\in\!\mathcal{L}(L^{p_0}(\Omega),L^{p_0}(\Omega))$ is 
	compact. Then $K\in\mathcal{L}(L^{p}(\Omega),L^{p}(\Omega))$,
        for all $p\in [p_0,p_1]$, and  $\sigma_{L^p}(K)$ is independent of ~$p$.
	\item Assume
          $K\in\mathcal{L}(L^{p_0}(\Omega),L^{p_1}(\Omega))$ is
          compact for some  $1\le p_0< p_1\le\infty$. Then $K\in\mathcal{L}(L^{p}(\Omega),L^{p}
	(\Omega))$, for all $ p\in[p_0,p_1]$, and $\sigma_{L^p}(K)$ is independent of ~$p$.
	\item Assume  $K\in	\mathcal
          {L}(L^{p_0}(\Omega),\mathcal{C}_b(\Omega))$ for some  $1\le
          p_0\le\infty $  is compact and	$X=\mathcal{C}_b
	(\Omega)$ or $X=L^r(\Omega)$ with $r\in[p_0,\infty]$. Then  $K\in\mathcal{L}
	(X,X)$, and $\sigma_{X}(K)$ is independent of $X$.
	\end{enumerate}
\end{prop}
\begin{proof}\noindent\\
	\indent{\it i.} Thanks to Proposition \ref{compacidad_cobos}, we have that 
	 $K\in \mathcal{L}(L^{p}(\Omega),L^{p}(\Omega))$ is compact
	for all $p\in[p_0,p_1]$. Thus the spectrum of $K$ is composed by  zero 
	and a discrete set of eigenvalues of finite multiplicity, (see \cite[chap. 
	6]{Brezis}). Let us prove now 
	that 	the eigenvalues of the spectrum $\sigma_{L^p(\Omega)}(K)$ are 
	independent of $p$.
	
        Now if $\lambda\in \sigma_{L^{p_1}}(K)$ is an
      eigenvalue, consider  an 	associated eigenfunction $\Phi\in L^{p_1}(\Omega)$. Since $\mu
	(\Omega)<\infty$ we have that   $\Phi\in L^p(\Omega)$  for all
        $p\in [p_0,p_1]$. Hence,  $\lambda\in \sigma_{L^{p}}(K)$ for all $p\in[p_0,p_1]$. 
	
	On the other hand,  if $\lambda\in \sigma_{L^{p}(\Omega)}(K)$ is an eigenvalue, with 
	$p\in [p_0,p_1)$, then any  associated eigenfunction $\Phi\in L^{p}(\Omega)$ 
	satisfies that 
	 \begin{equation}\label{eigenfunction_def}
		 K\Phi=\lambda \Phi.
	\end{equation}
	 Since $L^{p}(\Omega)\hookrightarrow L^{p_0}(\Omega)$  and $K:L^{p_0}(\Omega)\to L^{p_1}(\Omega)$, 
	 then $K\Phi \in L^{p_1}(\Omega)$. From \eqref{eigenfunction_def}, we 
	 obtain that $\Phi\in L^{p_1}(\Omega)$. Hence, $\lambda\in
         \sigma_{L^{p_{1}}(\Omega)}(K)$. 


	{\it ii.} We know that  $K\in\mathcal{L}(L^{p_0}(\Omega),L^{p_1}
	(\Omega))$ is compact,  and we have that 
	$$
		L^{p_1}(\Omega)\hookrightarrow L^{p_0}(\Omega)\stackrel{K}{\longrightarrow} L^
		{p_1}(\Omega) \hookrightarrow L^{p_0}(\Omega).
	$$
	Therefore $K\in \mathcal{L}(L^{p_1}(\Omega),L^{p_1}(\Omega))$ is compact, 
	and the hypotheses of Proposition \ref{compacidad_cobos} are satisfied. 
	Therefore  $K\in \mathcal{L}(L^{p}(\Omega),L^{p}(\Omega))$ is compact for all 
	$p\in[p_0,p_1]$. From part {\it i.}, we have the result.

	{\it iii.} We know that  $K\in\mathcal{L}(L^{p_0}(\Omega),\mathcal{C}_b(\Omega))$ is 
	compact. Since $\mu(\Omega)<\infty$, we have that for any
        $r\in[p_0,\infty]$ 
$$
\mathcal{C}_b (\Omega)\hookrightarrow L^r(\Omega)\hookrightarrow L^{p_0}(\Omega)\stackrel{K}{\longrightarrow} 
\mathcal{C}_b (\Omega)\hookrightarrow L^{r}(\Omega) \hookrightarrow
L^{p_0}(\Omega) 
$$

	Therefore, $K\in\mathcal{L}(X,X)$ is compact for 
	$X=\mathcal{C}_b(\Omega)$ or 
	$X=L^{r}(\Omega)$ with $r\in[p_0,\infty]$. Hence, following the 
	arguments in 
	 {\it i.} we have that $\sigma_{X}(K)$ is independent of  $X$.
\end{proof}



The following Proposition gives more details 
about the spectrum of $K_J$. 

\begin{prop} \label{sigma(K)}
	Let $(\Omega,\mu,d)$ be a metric measure space with $\mu(\Omega)
	<\infty$. We assume $K_J\in\mathcal{L}(L^{p_0}(\Omega),\mathcal{C}
	_b(\Omega))$ is compact for some  $p_0\le 2$. Let $X=L^p(\Omega)$, with 
	$p\in[p_0,\infty]$, or $X=\mathcal{C}_b(\Omega)$, and assume $J$ satisfies 
	that $$J(x,y)=J(y,x).$$
	Then $K_J\in\mathcal{L}(X,X)$ and $\sigma_{X}(K_J)\setminus \{0\}$ is a 
	real sequence of eigenvalues of 
	finite multiplicity, independent of $X$, that converges to $0$. 
	
	Moreover,  if we consider 
	\begin{equation}\label{M_y_m_intervalo}
	\begin{array}{lcr}
	m =\!\!\!\!\inf\limits_{u\in L^2(\Omega) \atop \|u\|_{L^2(\Omega)}=1}\!\!\!\!
	\langle K_Ju ,u
	\rangle_{L^2(\Omega)} &	 \mbox{and}   & M=\!\!\!\!\sup\limits_{u
	\in L^2(\Omega) \atop \|u\|_{L^2(\Omega)}	=1} \!\!\!\!\langle K_Ju ,u \rangle_{L^2
	(\Omega)}, 
	\end{array}
	\end{equation}	
	then $\sigma_{X}(K_J)\subset [m,M]\subset\mathbb{R},\; m\in
	\sigma_{X}(K_J)$ and $M\in\sigma_{X}(K_J)$. 
	In particular, $L^2(\Omega)$ admits an orthonormal basis consisting of 
	eigenfunctions of $K_J$.
\end{prop}
\begin{proof}
	Thanks to Proposition \ref	{adjoint_operator}, $K_J$ is selfadjoint in 
	$L^2(\Omega)$, then   $\sigma_{L^2}(K_J)\setminus \{0\}$ is a real 
	sequence of eigenvalues of finite 
	multiplicity that converges to $0$, (see  \cite[chap.6 ]{Brezis}). 
	Furthermore, from Proposition \ref{sigma_independiente} we have that $\,
	\sigma_{X}(K_J)\,$ is independent of $X$. Thus, the result.
	
	On the other hand, 
	we have that $\sigma_{X}(K_J)\subset [m,M]\subset\mathbb{R}$, with 
	$m \in\sigma_{X}(K_J)$ and $M\in\sigma_{X}(K_J)$, 
	 where $m$ and $M$ are given by \eqref{M_y_m_intervalo}, and 
	 thanks to the Spectral Theorem (see \cite[chap.6]{Brezis}), we know 
	 that $L^2(\Omega)$ admits an orthonormal basis consisting of 
	 eigenfunctions of $K_J$.
\end{proof}

 The following Corollary states that under the hypotheses of Proposition \ref
{K_aumenta_positividad}, any nonnegative eigenfunction associated to 
the operator $K_J$ is in fact strictly positive as well as its 
associated eigenvalue.
 \begin{corolario}
 	Let 
	$J$ satisfy the hypotheses of 	Proposition \ref{K_aumenta_positividad} 
	and assume $\Omega$ is $R$-connected. 
	If $\Phi\ge 0$, is an eigenfunction associated to an eigenvalue $\lambda$ of the operator 
	$K_J$, then $\Phi>0$, and the eigenvalue, $\lambda$, is also strictly positive.
 \end{corolario}
 \begin{proof}
 	Thanks to Proposition \ref{K_aumenta_positividad}, we know that, for 
	every function $\Phi\ge 0$, not identically zero defined in $\Omega$, it 
	happens 
	that $P(K_J^n(\Phi))\supset P_R^n(\Phi),\;\forall n\in\mathbb{N}$.
	
	On the other hand, since $\Phi$ is an eigenfunction associated to an 
	eigenvalue $\lambda$ of the operator $K_J$, we have that  
	$K_J^n(\Phi)=\lambda^n \Phi,\;\;\forall n\in\mathbb{N}$. Moreover, 
	from Proposition \ref{K_aumenta_positividad}, we know that for any 
	compact set  $\mathcal{K}\subset\Omega$, there 
	exists $n_0\in\mathbb{N}$ such that $P(K_J^n(\Phi))\supset\mathcal
	{K}$ for all $n\ge n_0$. Thus, $K_J^n(\Phi)=\lambda^n \Phi$ is strictly 
	positive  in $\mathcal{K}$ for all $n\ge n_0$. Therefore  
	$\Phi$ must be strictly positive in any compact set  $\mathcal{K}$ of 
	$\Omega$. Hence, $\lambda>0$ and $\Phi$ must be strictly positive 
	in $\Omega$.
 \end{proof}

%

Now, let us give some results about the spectral radius of the
operator $K$ 
 $$
 r(K)=\sup |\sigma(K)|.
 $$ 
For this,  we will use  Kre\u{\i{}}n-Rutman Theorem, \cite{krein_rutman}. 
We will work in the space 
$\mathcal{C}_b(\Omega)$, with $\Omega$ compact, 
and we consider the positive cone $C=\{f\in\mathcal{C}_b(\Omega); 
\;f\ge 0\}$, with Int$(C)=\{f\in \mathcal{C}_b(\Omega);\;f(x)>0,\;\forall x\in
\Omega \}
$. Thus, in the proposition below, we prove that the spectral radius of the 
operator $K$ is a simple eigenvalue that 
 has an associated eigenfunction that is strictly positive. 
\begin{prop}\label{spectral_radius_K}
	Let $(\Omega,\mu,d)$ be a metric measure space, with $\Omega$~  
	compact and connected. We assume that $J$ satisfies 
	$$J(x,y)=J(y,x)$$
	and
	\[J(x,y)>0,\;\forall x,\,y\in\Omega\;\; \mbox{such that}\;\; d(x,y)<R,\; 
	\mbox{for some} \;R>0,\] 
	and $K_J\in\mathcal{L}(L^p(\Omega),\mathcal{C}_b(\Omega))$, 
	for some  $1\le p\le\infty$, is compact, (see Proposition \ref{K_compact} 
	{\it ii.}).  
	
	Then $K_J\!\in\!\mathcal{L}(\mathcal{C}_b(\Omega),\mathcal{C}_b
	(\Omega))$ is compact, 
	the spectral radius $r_{\mathcal{C}_b(\Omega)}(K_J)$ is a positive 
	simple eigenvalue, and its associated eigenfunction can be
        taken  strictly positive.
\end{prop}
\begin{proof}
	 Since $\Omega$ is compact and connected 
	 then from Proposition \ref{K_aumenta_positividad} we obtain that, 
	 there 
	 exists $n_0 \in\mathbb{N}$ such that, for 
	 any nontrivial  nonnegative $u\in\mathcal{C}_b (\Omega)$,  
	 $\Omega= P_R^n(u)$, for all $n\ge n_0$, (see Definition \ref
	 {sets_P_n}), 
	and $\forall n\in
	 \mathbb{N}$, 
	 $P(K^nu )\supset P_R^n(u)$. Therefore $\Omega= P_R^n(u)\subset 
	 P(K^nu )$  for all $n\ge n_0$, i.e., for any nonnegative $u\in\mathcal
	 {C}_b(\Omega)$,  $K^n_Ju >0$ in  $\Omega$ 
	 for all $n\ge n_0$. Hence, $K_J$ is strongly positive in $\mathcal{C}
	 _b (\Omega)$. Moreover  
	 $K_J:\mathcal{C}_b(\Omega)\hookrightarrow L^p(\Omega)
	 \longrightarrow \mathcal{C}_b (\Omega)$ 
	 is compact. Hence, thanks to Kre\u{\i{}}n-Rutman Theorem, (see \cite
	 {krein_rutman}), the  spectral radius $r_{\mathcal{C}_b(\Omega)}(K_J)
	 $  is a positive  simple  eigenvalue with an  eigenfunction $\Phi$ 
	 associated  to it that is strictly positive. 
\end{proof}

A similar result was proved by Bates and Zhao \cite{bates_zhao}, for 
$\Omega\subset \R^N$ open,  but with the  stronger assumption  $J(x,y)>0$ for all $x,\,y\in\Omega$.


\subsection{Properties of  $K-hI$}

Let $(\Omega,\mu,d)$ be a metric measure space. We will always assume
below that

$\bullet$ If $X=L^p(\Omega)$, with $1\le p\le \infty$, we assume $h\in L^
{\infty}(\Omega)$.

$\bullet$ If $X=\mathcal{C}_b(\Omega)$, we assume $h\in \mathcal{C}_b
(\Omega)$.

The following result collects some properties of multiplication
operators. Note that below we denote $R(h)$ the range of the function
$h:\Omega \to \R$ and $\overline{R(h)}$, its closure. 

\begin{prop}\label{spectrum_resolvent_hI}
	Let  
	$h$ be  as above
        and consider  the multiplication operator $hI$, that maps
$$u(x)\mapsto h(x)u(x).$$
	Then the  resolvent set and spectrum  of the multiplication
        operator are independent of $X$ and are   given 
	by 
	$$\rho_X(hI)=\mathbb{C}\setminus \overline{R(h)}, \quad \sigma_X(hI)=\overline{R(h)},$$
	respectively. Moreover, for $X=L^p(\Omega)$, 
	 the eigenvalues 	associated to 
	$hI$ have infinite multiplicity and satisfy 
	\[EV(hI)=\left\{\alpha\; ;\; \mu\left(\left\{x\in\Omega\;;\; h(x)=\alpha\right\}
	\right)
	>0 \right).\]
	
	On the other hand, for $X=\mathcal{C}_b(\Omega)$, 	the
        eigenvalues have infinite multiplicity and satisfy 
        \begin{displaymath}
          	EV(hI)\!\supset \!\left\{\alpha;\ \left\{x\in\Omega\;;\;
          h(x)=\alpha\right\} \; \mbox{has nonempty interior}\right\}
        \end{displaymath}
\end{prop}
\begin{proof}
	If $X=L^p(\Omega)$ consider $f\in L^p(\Omega)$ and $u\in L^p	(\Omega)$,
        such that $h(x)u(x)-\lambda u(x)  = f(x)$, that is, 
        \begin{displaymath}
          u(x)  = \displaystyle{ \frac{f(x)}{h(x)-\lambda}=\frac{1}{h(x)-\lambda}\,f(x)}.
        \end{displaymath}
	Then we have that $\lambda\!\in\!\rho_{L^p(\Omega)}(hI)\,$ if and 
	only if $(hI-\lambda I)^{-1}\in  
	\mathcal{L}(L^p(\Omega))$, if and only if $\;\displaystyle
	{\frac{1}{h-\lambda}\in L^{\infty}(\Omega)}$, and this happens
%
if and only if $\lambda\not\in \overline{\mbox{R}(h)}$. Thus,
$\rho_{L^p(\Omega)}(h)=\mathbb{C}\setminus \overline{\mbox{R}(h)}$.
%
	
	If $\lambda$ is an  eigenvalue, then for some  $\Phi\in L^p
	(\Omega)$ with $\Phi\not\equiv 0$, we have 
	\[h(x)\Phi(x)=\lambda \Phi(x)\]
	and this only happens if there exists a set $A\subset\Omega$, with $\mu(A)
	>0$, such that $h(x)=\lambda$ for all $x\in A\subset\Omega$. 
	Hence, we have that $\mbox{Ker}(hI-\lambda I)= L^{p}(A)$. Thus, the result.
	
        If $X=\mathcal{C}_b(\Omega)$ the
          arguments run as above. Just note that if $\{\lambda;\ \left\{x\in\Omega\;;\;
          h(x)=\lambda\right\}$ has nonempty interior, $A$,  then
        $\mbox{Ker}(hI-\lambda I)= \{\Phi\in\mathcal{C}_b(\Omega): 
	\Phi(x)=0, \;\forall x\in\Omega\setminus A\}$. 
\end{proof}



Now  we consider the particular case of the function 
$$
 h_0(x)=\int_{\Omega}J(x,y)dy.
$$
for which we assume $J \in L^{\infty}(\Omega, L^{1}(\Omega))$ and
hence   $h_0\in L^{\infty}(\Omega)$. 


%
Then we have 

 \begin{prop}\label{Green's formulas}{\bf (Green's formulas)}$\;$
	Let  
	$\mu(\Omega)<\infty$. 
	Assume $J \in L^{\infty}(\Omega, L^{1}(\Omega)) \cap
        L^p(\Omega,L^{p'}(\Omega))$, for some  $1\le p <\infty$, and   
	 \begin{displaymath}  
	 	J(x,y)=J(y,x). 
	 \end{displaymath} 
	 Then for $u\in L^p(\Omega)$ and $\;v\in L^{p'}(\Omega)$,
	\begin{equation}\label{green_eq_2}
		\langle K_Ju -h_0Iu , v\rangle_{L^p,L^{p'}}\!=\!\displaystyle -\frac{1}
		{2}\!\int_{\Omega}\!\int_{\Omega}\!\!J(x,y)(u(y)-u(x))(v(y)-v(x))dy\,dx.
	\end{equation}
	In particular, if $p=2$ we have that for $u\in L^2(\Omega)$
	\begin{displaymath} 
		\langle K_Ju -h_0\,Iu ,\, u\rangle_{L^2,L^2}=\displaystyle -\frac{1}
		{2}\int_{\Omega}\int_{\Omega}J(x,y)(u(y)-u(x))^2 dy\,dx.
	\end{displaymath}
 \end{prop}
 \begin{proof} 
Observe that 	
 \[ \begin{array}{ll}
	\displaystyle I_1 &\displaystyle =\int_{\Omega}\displaystyle \int_{\Omega}J
	(x,y)
	(u(y)-u(x))(v(y)-v(x))dy\,dx
	\smallskip \\
	 &\displaystyle=\!\int_{\Omega}\!\int_{\Omega}\!\!J(x,y)(u(y)\!-\!u(x))v(y)dydx-\!\!\int_
	{\Omega}\!\int_{\Omega}\!\!J(x,y)(u(y)\!-\!u(x))v(x)dydx.
	\end{array}
	\]
	Relabeling variables in the first term above, we obtain
	\[
	I_1\!\!=\!\!\int_{\Omega}\!\int_{\Omega}\!\!J(y,x)(u(x)-u(y))v(x)dx\,dy-\!
	\int_{\Omega}\!\int_{\Omega}\!\!J(x,y)(u(y)-u(x))v(x)dy\,dx.
	\] 
	Now, since $J(x,y)=J(y,x)$, 
	\[
	I_1=\!\!\int_{\Omega}\!\int_{\Omega}\!\!J(x,y)(u(x)-u(y))v(x)dxdy-\!\int_{\Omega}\!
	\int_{\Omega}\!\!J(x,y)(u(y)-u(x))v(x)dydx.
	\] 
	Thanks to Fubini's Theorem, we have 	that 
        \begin{equation} \label{aux_I1}
\begin{array}{ll}	
I_1 & \displaystyle =-2\int_{\Omega}\!\int_{\Omega}J(x,y)(u(y)-u(x))v(x)dy\,dx.
	\end{array}
        \end{equation}

	
	 On the other hand, thanks to the hypothesis on $J$, $h_0\in L^{\infty}
	 (\Omega)$ and from Propositions \ref{prop1}  
	 we  have that $K_J-h_0I\in\mathcal{L}(L^p(\Omega))$, 
	 for all $1\le p\le\infty$. 
	Hence, if $u\in L^p(\Omega)$ and  $v\in L^{p'}(\Omega)$
	 \begin{equation}\label{aux2}
	 \begin{array}{ll}
	 \!\!\!\!\displaystyle \langle K_Ju \!-\!h_0\,Iu , v\rangle_{L^p,L^{p'}}\!\!\!\!
	 &\!\! =\!\!
	 \displaystyle\int_{\Omega}\!\!\left(\int_
	 {\Omega}\!\!\!J(x,y)u(y)dy\!-\!\int_{\Omega}\!\!\!J(x,y)dy\,u(x)\!\right)\!v(x)
	 \,dx
	 \smallskip\\
	 & \displaystyle=\int_{\Omega}\int_
	 {\Omega}J(x,y)(u(y)-u(x))v(x)dy\,dx.
	 \end{array}
	 \end{equation}
	 Hence, from \eqref{aux_I1} and \eqref{aux2}, we obtain  
	 \eqref{green_eq_2}. 
	 The second part of the proposition is an immediate consequence of 
	 \eqref{green_eq_2}.
 \end{proof}



Now  we describe the spectrum of $K-hI\in\mathcal{L}(X,X)$, where
$X=L^p(\Omega)$, with $1\le p\le \infty$, or
$X=\mathcal{C}_b(\Omega)$, and we prove that,  
under certain conditions on the operator $K$,  it is 
independent of $X$. Moreover, we give conditions on 
$J$ and $h$ under which the spectrum of $K_J-hI$ is nonpositive. For
this, we  start introducing some definitions used in the following
theorems. 
\begin{definition}
	If $T$ is a linear operator in a Banach space $Y$, a {\bf normal point} of $T$ 
	is any complex number which is in the resolvent set, or is an isolated 
	eigenvalue of $T$ of finite multiplicity. Any other complex number is in the 
	{\bf essential spectrum} of $T$.
\end{definition}

To describe the spectrum of $K-hI$, we use the following theorem that can be found 
in \cite[p. 136]{Henry}.

\begin{teo}\label{Henry}
	Suppose $Y$ is a Banach space, $T:D(T)\subset Y\rightarrow Y$ is a 
	closed linear operator, $S: D(S)\subset Y\rightarrow Y$ is linear 	with 
	$D(S)
	\supset D(T)$ and $\;S(\lambda_0-T)^{-1}$ is compact for some ~$
	\lambda_0\in\rho(T)$. Let $\;U$ be an open connected set in 
	$\mathbb{C}$ 
	consisting entirely of normal points of $\;T$, which are points of the 
	resolvent of $\;T,$ or isolated eigenvalues of $T$ of finite multiplicity. 
	Then either $U$ consists entirely of normal points of $\;T+S,$ or 
	entirely of eigenvalues of $\;T+S$.
\end{teo}
\begin{remark}
	If $S:Y\rightarrow Y$ is compact, Theorem \ref{Henry} implies that the 
	perturbation $S$ can not change the essential spectrum of $T$.
\end{remark}
The next theorem describes the spectrum of the operator $K-hI$ in 
$X$. Recall  that if $X=L^p(\Omega)$, with $1\le p\le \infty$, we
assume 	 $h\in L^{\infty}(\Omega)$, while if
$X=\mathcal{C}_b(\Omega)$, we assume   	$h\in\mathcal{C}_b(\Omega)$.

 \begin{teo}\label{espectro_K-h}
	 If $K\in\mathcal{L}\left(X,X\right)$ is compact,  (see Proposition \ref
	 {K_compact}), then
	 $$\sigma(K-hI)=\overline {R(-h)}\cup \left\{\mu_n\right\}_{n=1}^{M},\; 
	 \mbox{ with }\;M\in\mathbb{N}\cup\{\infty\}.$$
	 If $M=\infty$, then $\; \left\{\mu_n\right\}_{n=1}^{\infty}\;$ is a sequence 
	 of eigenvalues of $K-hI$ with finite  multiplicity, that accumulates 
	 in $\overline {R(-h)}$.
 \end{teo}

\begin{proof}
	 With the notations of Theorem \ref{Henry}, we consider the operators 
	 $$\,S=K\;\mbox{ and }\; T=-hI.$$ 
	 First of all, we prove that $\mathbb{C}\setminus\overline {R (-h)}\subset 
	 \rho(K-hI)$. We choose the set $U$ in Theorem \ref{Henry} as
	 $$U=\rho(-hI)=\rho(T)=\mathbb{C}\setminus\overline {R (-h)}$$
         which is an   open, connected set. Since $U=\rho(T)$, 
	every $\lambda\in U$ is a normal point of $T$. 

	 On the other hand, if $\lambda_0\in\rho(T)$, then $(T-\lambda_0)^{-1}\in
	 \mathcal{L}(X, X)$, and $S=K$ is compact. Then, we 
	 have that $S(\lambda_0-T)^{-1}\in\mathcal{L}(X,X)$ 
	 is compact. Thus, all the hypotheses of Theorem \ref{Henry} are satisfied. 
	 Now, thanks to Theorem \ref{Henry}, we have that $U=\mathbb{C}\setminus
	 \overline {R(-h)}$ consists entirely of eigenvalues of $T+S=K-hI$ or $U$ 
	 consists entirely of normal points of $T+S=K-hI$.
	 
	If $U=\mathbb{C}\setminus\overline {R(-h)}$ consists entirely of 
	eigenvalues of $T+S=K-hI$, we arrive to contradiction, because the 
	spectrum of $K-hI$ is bounded. So $U=\mathbb{C}\setminus\overline 
	{R(-h)}$ has to consist entirely of normal points of ~$T+S$. Then, they are 
	points of the resolvent or isolated eigenvalues of $\;T+S=K-hI$. Since any 
	set of isolated points in $\mathbb{C}$ is a finite set, or a numerable set, 
	we 	have that the isolated eigenvalues are 
	 $$\left\{\mu_n\right\}_{n=1}^{M},\; \mbox{ with }\;M\in\mathbb{N}\; 
	 \mbox{ or } \;M=\infty.$$
	 Moreover, since the spectrum of $K-hI$  is bounded, if $M=\infty$ then 
	 $\;  \left\{\mu_n\right\}_{n=1}^{\infty}\;$ is a  sequence of eigenvalues 
	 of $K-hI$ with finite multiplicity, that accumulates  in $\overline {R(-h)}$. 
	  
	Now we prove that $\overline{R(-h)}\subset\sigma(K-hI)$. We argue 
	by contradiction. Suppose that there exists a $\tilde{\lambda}\in 
	\overline {R(-h)}$ that belongs to $\rho (K-hI).$ Since the resolvent set 
	is open, 
	there exists a ball $B_{\varepsilon}(\tilde{\lambda})$ centered in $\tilde
	{\lambda},$ that is into the resolvent of $K-hI$.
	 Then $U=B_{\varepsilon}(\tilde{\lambda})$ is an open connected set that 
	 consists of normal points of $K-hI$. With the notation 
	 of Theorem  \ref{Henry}, we consider  
	 the operators
	 $$T=K-hI\;\;\mbox{ and }\;\;S=-K $$
	and the open, connected set 
	$$U=B_{\varepsilon}(\tilde{\lambda}).$$
	 Arguing like in the previous case, if $\lambda_0\in\rho(T)$, we have that $S
	 (\lambda_0-T)^{-1}$ is compact, thus the hypotheses of  Theorem \ref
	 {Henry} are satisfied. Hence
	 $U=B_{\varepsilon}(\tilde{\lambda})$ consists entirely of eigenvalues of 
	  $T+S=-hI$ or $U=B_{\varepsilon}(\tilde{\lambda})$ consists entirely of normal 
	 points of $T+S=-hI$.
	 
	 If $U=B_{\varepsilon}(\tilde{\lambda})$ consists 
	 entirely of eigenvalues of $T+S=-hI$, we would arrive to contradiction, 
	 because the eigenvalues of $-hI$ are only inside $\overline
	 {R(-h)}$,  and 
	 the ball $B_{\varepsilon}(\tilde{\lambda})$ is not inside $
	 \overline {R(-h)}$. 
	 So $U= B_{\varepsilon}(\tilde{\lambda})$ has to consist of 
	 normal points of  $T+S=-hI$, so they are points of the resolvent of $-hI
	 $ or 
	 isolated eigenvalues of 
	 finite multiplicity of $-hI$. Since $\rho(-hI)=\mathbb{C}\setminus \overline {R
	 (-h)}$, and  $\tilde{\lambda}\in \overline {R(-h)}$, we have that $\tilde
	 {\lambda}$ has to be an isolated eigenvalue of $-hI$, with finite multiplicity. 
	 But from Proposition \ref{spectrum_resolvent_hI}, we know that the 
	 eigenvalues of $-hI$ have infinity multiplicity. Thus, we arrive to  
	 contradiction. Hence, we have proved that $\overline{R(-h)}\subset\sigma(K-
	 hI)$. With this, we have finished the proof of the theorem.
\end{proof}


In the following proposition we give conditions that guarantee  that  the spectrum of 
$K-hI$ is independent of $X=L^p(\Omega)$ with $1\le p\le \infty$, or $X=
\mathcal{C}_b(\Omega)$.
\begin{prop}\label{independent_spectrum_K_h}
	Let 
	$\mu(\Omega)<\infty$. 
	\begin{enumerate}
	\renewcommand{\labelenumi}{\roman{enumi}.}
	\item Assume,  for some  $1\!\le\! p_0\!<\! p_1\!<\!\infty$,
          $K\!\in\!\mathcal{L}(L^{p_0}(\Omega),L^{p_1} 
	(\Omega))$, $K\!\in\!\mathcal{L}(L^{p_0}(\Omega),L^{p_0}
	(\Omega))$ is 
	compact and $h\in L^{\infty}(\Omega)$. Then  $K-hI\in\mathcal{L}(L^{p}
	(\Omega),L^{p}(\Omega))$, $\forall p\in
	[p_0,p_1]$, and  $\sigma_{L^p}(K-hI)$ is independent of $\;p\;$.
	\item Assume, for some  $1\le p_0< p_1\le\infty$,  $K\in\mathcal{L}(L^{p_0}(\Omega),
	L^{p_1}	(\Omega))$ is compact  and $h\in
        L^{\infty}(\Omega)$. Then  $K-hI\in\mathcal
	{L}(L^{p}(\Omega),L^{p}(\Omega))$, $\forall p\in[p_0,p_1]$, and 
	$\sigma_{L^p}
	(K-hI)$ is independent of ~$p$.
	\item Assume, for  some  $1\le p_0\le\infty $,
          $K\in\mathcal{L}(L^{p_0}(\Omega),  
	\mathcal{C}_b(\Omega))$ is compact and $X=\mathcal{C}_b
	(\Omega)$ or $X=L^r(\Omega)$ with $r\in[p_0,\infty]$, and $h\in 
	\mathcal{C}_b(\Omega)	$. Then  $K-hI\in\mathcal{L}(X,X)$ and
        $\sigma_{X}(K-hI)$ is independent  
	of $X$.
	\end{enumerate}
\end{prop}
\begin{proof}
	Following the same arguments in Proposition \ref{sigma_independiente}, we 
	have that in any of the cases {\it i.}, {\it ii.}, or {\it iii.},  $K\in\mathcal{L}(X,X)$ 
	is compact, where $X=L^p(\Omega)$ with $p_0\le p\le p_1$ for the cases {\it i.} and {\it ii.}, and $X=L^p(\Omega)$ with $p_0\le p\le \infty$, or $X=
\mathcal{C}_b(\Omega)$ for the case {\it iii.}. 
Then, from Theorem \ref{espectro_K-h} we have that 
	$$\sigma_{X}(K-hI)=
	\overline{R(-h)}\cup \left\{\mu_n\right\}_{n=1}^{M},\;\mbox{ with }\; M\in
	\mathbb{N} \;\mbox{ with }\; or M=\infty,$$ 
	where $\left\{\mu_n\right\}_{n}$ are eigenvalues of $K-hI$, with finite multiplicity 
	  $\forall n \in
	\{1,\dots,M\}$.
	
	Since $\overline{R(-h)}$ is independent of $X$,
	we just have to prove that the eigenvalues $\lambda\in\sigma_{X}(K-hI)$ 
	satisfying that $\lambda\notin\overline{R(-h)}$ are independent of $X$. 
	Let $\lambda\in\sigma_{X}(K-hI)$ be an eigenvalue such that $\lambda\notin
	\overline{R(-h)}$. We denote by $\Phi$ an eigenfunction associated to 
	$\lambda\in\sigma_{X}(K-hI)$, then
	\begin{equation}\label{equality_eigenfunstion_spectr_K_h}
		K\Phi(x)-h(x)\Phi(x)=\lambda \Phi(x)
	\end{equation}
	Since $\lambda\notin\overline{R(-h)}$, then from  
	\eqref{equality_eigenfunstion_spectr_K_h} we obtain
	\begin{equation}\label{equality_eigenfunstion_spectr_K_h_2}		
	\begin{array}{rcl}
	\Phi(x)&= & \displaystyle\frac{1}{h(x)+\lambda}K\Phi(x)
	\end{array}
	\end{equation}
	and  $\frac{1}{h
	(\cdot)+\lambda}\in L^{\infty}(\Omega)$.

      For the cases  {\it i.}  and {\it ii.}, thanks to the hypotheses on $K$, we have 
	\begin{equation}\label{equality_eigenfunstion_spectr_K_h_3}
		\displaystyle\frac{1}{h(\cdot)+\lambda}K \in\mathcal{L}(L^{p_0}
		(\Omega),L^{p_1}(\Omega)).
	\end{equation}

Now assume $\lambda\in \sigma_{L^{p_1}}(K-hI)$ is  an
      eigenvalue with associated eigenfunction
      $\Phi\in L^{p_1}(\Omega)$. Since $\mu 
	(\Omega)<\infty$ we have that  $L^{p_1}(\Omega)
	\hookrightarrow L^p(\Omega)$ for all $p\le p_0$, then $\Phi\in
        L^p(\Omega)$ and $\lambda\in \sigma_{L^{p}}(K-hI)$ for all $p\in
	[p_0,p_1]$. 
	
On the other hand, if  $\lambda\in \sigma_{L^{p}(\Omega)}(K-hI)$ is an eigenvalue, with 
	 an  associated eigenfunction $\Phi\in L^{p}(\Omega)$, with
         $p\in [p_0,p_1)$, then by
         $L^{p}(\Omega)\hookrightarrow L^{p_0}(\Omega)$ and 
         \eqref{equality_eigenfunstion_spectr_K_h_3}, we have that  
	 $\frac{1}{h(\cdot)+\lambda}K\Phi \in L^{p_1}(\Omega)$. Hence, from 
	 \eqref{equality_eigenfunstion_spectr_K_h_2}, we obtain that 
	 $\Phi\in L^{p_1}(\Omega)$. Therefore, $\Phi\in L^p(\Omega)$ for $p\in 
	 [p_0,p_1]$, and  the spectrum is independent of the space
         in  cases {\it i.}  and {\it ii.}.
	 	 
	 The case {\it iii.} is analogous to the previous result, using that  
	  $h\in \mathcal{C}_b(\Omega)$ and $\lambda\not\in 
	 \overline{R(-h)}$, then
	\begin{displaymath}  
		\displaystyle\frac{1}{h(\cdot)+\lambda}K\Phi  \in\mathcal{L}(L^{p_0}
		(\Omega),\mathcal{C}_b(\Omega)).
	\end{displaymath}
	 Thus, the result.
\end{proof}


The following results give conditions for  the spectrum of 
$K_J-hI$ to be  nonpositive. 

\begin{corolario}\label{spectrum_negative}
	Let  
	$\mu(\Omega)<\infty$. Assume for some  $1\le p_0\le 2$,
        $K_J\in\mathcal{L}(L^{p_0}(\Omega),X)$ is compact with
        $X=L^p(\Omega)$, with $p_0\le p\le \infty$, or
        $X=\mathcal{C}_b(\Omega)$ 
%
  and $J$ satisfies that 
	$$J(x,y)=J(y,x) . $$
	Then 
	\begin{enumerate}
	\renewcommand{\labelenumi}{\roman{enumi}.}
	\item If $h\equiv c $, with $c\in \mathbb{R}$ such that $c>r(K_J)$, 
	where $r(K_J)$ is the spectral radius of $K_J$ then $\sigma_X(K_J-hI)
	$ is real and 	nonpositive.
	\item If $J\in L^{\infty}(\Omega, L^1(\Omega))$ and
          $h(x)=h_0(x)=\displaystyle\int_{\Omega}J(x,y)dy\in L^{\infty} 
	(\Omega)$,  	satisfies that  
	$h_0(x)\geq \alpha>0$ for all $x\in \Omega$, then $\sigma_X(K_J-h_{0}I)$ is 
	nonpositive and $0$ is an isolated eigenvalue with finite 
	multiplicity. Moreover if $J$ satisfies that 
	$$J(x,y)>0,\;\forall x,\,y\in\Omega\; \mbox{such that}\; d(x,y)<R$$
	and $\Omega$ is $R$--connected, then $\{0\}$ is a simple
        eigenvalue with only constant eigenfunctions. 
	
	\item If $h\in L^{\infty}(\Omega)$ satisfies that $\;h\ge h_0\,$ in $\,
	\Omega$,
	then $\sigma_X(K_J-hI)$ is nonpositive.
	\end{enumerate}
\end{corolario}
\begin{proof} Under the hypotheses and thanks to the previous Proposition 
	\ref{independent_spectrum_K_h}, we have that $\sigma_X(K-hI)$ is 
	independent of $X$. Hence the rest of the results will be proved in 
	$L^2(\Omega)$.
	
	\indent{\it i. } From Proposition \ref{adjoint_operator} we
        have that $K_J$ is selfadjoint in $L^2(\Omega)$, and so is
        $K_J-hI$.  Thus  $\sigma_{L^2(\Omega)}(K_J)$ is composed by real 
	values that are less or equal to $r(K_J)$,  (see \cite[p.165]{Brezis}).
	
	On the other hand,  $\sigma_{L^2(\Omega)}(K_J-hI)=\sigma_{L^2(\Omega)}
	(K_J)-c$ 
	and $c>r(K_J)$, then	 we have that $\sigma_{L^2(\Omega)}(K_J-hI)$ is real 
	and nonpositive. 
	
	{\it ii. } Under the hypotheses we have that  $K\in\mathcal{L}(X,X)$ is 
	compact, then thanks to 
	Theorem \ref{espectro_K-h}, we know that 
	 $$\sigma_X(K-h_0I)=\overline{R(-h_0)}\cup\left\{\mu_n\right\}_{n=1}^M,\quad \mbox
	{with } M\in\mathbb{N}\;\mbox{ or } M=\infty.$$ 
%
        Then, thanks to Proposition \ref{Green's formulas},
        \begin{equation} \label{sign_012_1}
           \!\!\!\!\!\langle(K_J\!-\!h_0)u,u \rangle_{L^2,L^2} 
	=  {-\frac{1}{2}\int_{\Omega}\int_{\Omega}J(x,y)(u(x)-u(y))^2dy
	\,dx}
	\le0,
        \end{equation}
From this we get 
	\begin{displaymath}  
	\sigma_{L^2(\Omega)}(K_J-h_0)\le\sup\limits_{u\in L^2(\Omega) \atop \|u\|_
	{L^2(\Omega)}=1}\langle(K_J-
	h_0)u,u\rangle_{L^2(\Omega),L^2(\Omega)}\le0.
	\end{displaymath}
	
	Observe that clearly constant function satisfy $(K-h_{0}) u=0$
        and since $0\not\in\overline{R(-h_0)}$, then $0$ is an isolated eigenvalue 
	with finite multiplicity. 
	Let us prove below that $\{0\}$ is a simple eigenvalue. We consider 
	 $\varphi$ an eigenfunction associated to $\{0\}$. Thanks to 
	 Proposition \ref{Green's  formulas} in $L^2(\Omega)$ we have
	 \begin{equation*}
	 0=\langle (K-h_0\,I)\varphi,\, \varphi\rangle_{L^2(\Omega),L^{2}
	 (\Omega)}=	\displaystyle -\frac{1}{2}\int_{\Omega}\int_{\Omega}J(x,y)
	 (\varphi(y)-\varphi(x))^2dy\,dx.
	\end{equation*}
	Since $J(x,y)>0,\;\forall x,\,y\in\Omega\; \mbox{such that}\; d(x,y)<R$, 
	 then for all 
	$x\in\Omega$, $\varphi(x)=\varphi(y)$ for any $y\in
        B(x,R)$. Thus, since $\Omega$ is $R$--connected,  
	$\varphi$ is a 
	constant function in $\Omega$. Therefore,  $\{0\}$ 	is  simple.

	{\it iii. } If  $h \ge h_0$ from  \eqref{sign_012_1}, we have
        \begin{displaymath}
          	\langle(K_J-hI)u,u\rangle_{L^2(\Omega),L^2(\Omega)} =
                \langle(K_J-h_0)u,u\rangle_{L^2(\Omega),L^2(\Omega)}+\langle(h_0\!-\!h)u,u 
	\rangle_{L^2(\Omega),L^2(\Omega)}
	\le0.
        \end{displaymath}
\end{proof}




\section{The linear evolution equation}
\label{sec:line-evol-equat}

Let  $(\Omega,\mu,d)$ be a metric 
measure space. Let 
$X=L^p(\Omega)$, with $1\le p\le\infty$ or $X=\mathcal{C}_b(\Omega)$.  
The problem we are going to work with in this 
section, is the following
\begin{equation}\label{eq1}
 \left\{
 \begin{array}{lll}
 u_t(x,t) & =(K_{J}-hI)u (x,t)=Lu (x,t),& x\in\Omega,\, t>0 \\
 u(x,0) & =u_0(x), & x\in\Omega,  
 \end{array}
 \right.
 \end{equation}
 with  $K_Ju(x) =\int_{\Omega}J(x,y)u(y)dy$,  $J \geq 0$, 
$\;u_0\in X$, $K_J\in\mathcal{L}(X,X)$ and if $X=L^p(\Omega)$,
 with $1\le p\le\infty$, we assume $h\in L^{\infty}(\Omega)$ while if
 $X=\mathcal{C}_b(\Omega)$, we assume $h\in\mathcal{C}_b(\Omega)$.

First, since $K_J\in\mathcal{L}(X,X)\;$ then the problem \eqref{eq1} has a
unique strong solution $u \in\mathcal{C}^{\infty}(\mathbb{R},X)$,
given by
\begin{displaymath}
  u(t)=e^{Lt}u_0.
\end{displaymath}
The mapping
$$
	\mathbb{R}\ni t\;\mapsto\; u(t)=e^{Lt}u_0\in X
$$
is analytic. Moreover the mapping $(t,u_0)\;\mapsto\; e^{Lt}u_0$ is continuous.

	We denote the group associated to the operator $L=K_J-hI$ with $S_{K,\,h}$, to 
remark the dependence on $K_J$ and $h$. Hence the solution 
of \eqref{eq1} is
\begin{displaymath}  
	u(t, u_0)=S_{K,\,h}(t)u_0=e^{Lt}u_0.
\end{displaymath}

 


\subsection{Maximum principles} 

 
 First of all, let us consider the problem \eqref{eq1}, with $h\equiv 0$,
  	\begin{equation}\label{eq_for_fomal_proof}
	\left\{
	\begin{array}{rl}
	\displaystyle\frac{du}{dt}=&K_Ju ,
	\smallskip\\
	u(0)=&u_0\ge 0.
	\end{array}
	\right.
	\end{equation}
Then the solution to \eqref{eq_for_fomal_proof} 
can be written as 
	\begin{displaymath}  
	u(x,t)=e^{K_Jt}u_0(x)=\left(\sum_{k=0}^\infty \frac{t^kK_J^k}{k!}\right) 
	u_0(x).
	\end{displaymath}
Since $J$ is nonnegative, we have that $K_J^ku_0$ is nonnegative for any $u_0$ 
nonnegative, $\forall k\in\mathbb{N}$. Then we have that the solution $u(x,t)$ is 
nonnegative. In fact, for any $m\ge 0$
$$
\begin{array}{c}
\displaystyle u(x,t)\!\ge\! u_0(x)\!\ge\! 0,\;\; u(x,t)\!\ge\! u_0(x)+tK_Ju_0(x)\!
\ge\! 0,\,\mbox{ and }\,u(x,t)\!\ge\! \left(\sum\limits_{k=0}^m \frac{t^kK_J^k}{k!}\right)u_0(x)\!\ge\! 0.
\end{array}
$$ 

 Now, for $h\not\equiv 0$, let  $u$ be the solution to \eqref{eq1}. We take the function 
 $$v(t)=e^{h(\cdot)t}u(t),\;\mbox{ for }\, t\ge 0.$$ 
This function $v$ satisfies that
%
$$
v_t(x,t)=e^{h(x)t}K_Ju (x,t), \;\mbox{ and }\; v(x,0)=u_0(x),
$$
%
hence, integrating in time we get  
\begin{equation}\label{Form_Var_Const}
 u(x,t)=e^{-h(x)t}u_0(x)+\displaystyle\int_{0}^t e^{-h(x)(t-s)}K_Ju (x,s)ds.
 \end{equation}
\noindent
Let $X=L^p(\Omega)$, with $1\le p\le\infty$, or $X=\mathcal{C}_b
(\Omega)$. For every $\omega_0\in X$ and $T>0$, we consider the 
mapping $\mathcal{F}_{\omega_0}:\mathcal{C}([0,T];X)\rightarrow 
\mathcal{C}([0,T];X)$ defined  as 
 \[ \mathcal{F}_{\omega_0}(\omega)(x,t)  =e^{-h(x)t}\omega_0(x)+\displaystyle\int_{0}
^t e^{-h(x)(t-s)}K_J(\omega)(x,s)ds. \]


Then we have the following immediate result.  
 \begin{lema}\label{F_u_0_contrative}
	If    $\omega_0,\;z_0\in X$, and $\,\omega,\;z
	\in X_T=\mathcal{C}([0,T];X)$,  then there exist two constants $C_1$ 
	and $C_2$ 
	 depending on $ h$ and $T$, such that 
	 \begin{equation}\label{contraction}
		||| \mathcal{F}_{\omega_0}(\omega)-\mathcal{F}_{z_0}(z)|||\le C_1
		(T)\|
		\omega_0-z_0\|_{X}+C_2(T)|||\omega-z|||, 
	\end{equation}
	 where $C_1(T)=e^{\|h_-\|_{L^{\infty}(\Omega)}T}$, ~$C_2(T)=CTe^{\|
	 h_-\|_{L^{\infty}(\Omega)}T}$,  
	 $C_2:[0,\infty)\to \R$ is increasing and continuous, and $C_2(T)\to 0$, 
	 as $T\to 0$.
\end{lema}
  

With this we can prove the following.

\begin{prop}{\bf (Weak maximum principle)} \label{positive_solution_with_u_0_positive}
%
	 For every nonnegative $u_0\in X$, the solution to the 
	 problem  \eqref{eq1} is nonnegative for  all $t\ge 0$. 
 \end{prop}
 \begin{proof}
	 Thanks to \eqref{Form_Var_Const}, we know that the solution to \eqref{eq1} can 
	 be written as
	 \begin{displaymath} 
	 u(x,t)=e^{-h(x)t}u_0(x,t)+\displaystyle\int_{0}^t e^{-h(x)(t-s)}K_Ju (x,s)ds=
	 \mathcal{F}_{u_0}u (x,t).
	 \end{displaymath}
	 
	 We choose $T$ small enough such that $C_2(T)$ in Lemma \ref
	 {F_u_0_contrative} satisfies that $C_2(T)<1$. Hence, by \eqref
	 {contraction} we have that  
	 $\mathcal{F}_{u_0}(\cdot)$ is a contraction in $X_T=\mathcal{C}
	 \big([0,T];X\big)$. 
	  We consider the sequence of Picard iterations,
	 \[u_{n+1}(x,t)=\mathcal{F}_{u_0}(u_n)(x,t)\;\;\forall n\ge1,\; x\in\Omega,\; 
	 0\le t\le T.\]
	 Then the sequence $u_n$ converges to $u$ in $X_T$. 
	 We take  $\;u_1(x,t)=u_0(x)\ge 0$, then for $t\ge 0$
	 \begin{displaymath}  
		u_2(x,t)=\mathcal{F}_{u_0}(u_1)(x,t)=e^{-h(x)t}u_0(x)+\displaystyle\int_{0}^t e^
		{-h(x)(t-s)}K_J(u_0)(x)ds
	\end{displaymath}
	is nonnegative, because $K_J$ is a positive operator. Thus $u_2(x,t)\ge 
	0$ for all $t\ge 0$. 
	 Repeating  this argument for all $u_n$,  we get that $u_n(x,t)$ is 
	 nonnegative for 
	 every $n\ge1$, for $t\ge 0$. 
	 As $u_n$ converges to $u$ in $X_T$, we have that $u$ is 
	 nonnegative.
	 
	 
	 Since $T>0$ does not depend on the initial data, if we
         consider again the same problem with initial data
         $u(\cdot,T)$,  
	 then the solution $u(\cdot, t)$ is nonnegative for all $t\in[T,\,
	 2T]$. 
	 Since 
	 \eqref{eq1} has a unique solution then 
	 we have proved that the solution  of \eqref{eq1}, $u(x,t)\ge 0$ for all $t\in [0, 2T]$. Repeating this 
	 argument, we have that the solution of \eqref{eq1} is nonnegative $\forall t\ge 0$.
%
 \end{proof}
 
 

Now we show that with the assumptions in Proposition
\ref{K_aumenta_positividad} we have in fact the strong maximum
principle. 

\begin{teo}{\bf (Strong maximum principle)}\label
	{linear_solution_strictly_pos}
	Assume $  K_J\in \mathcal{L}(X,X)$, and $J\ge 0$ satisfies 
	\[J(x,y)>0\, \mbox{ for all } \,x,\,y\in\Omega,\, \mbox{ such that } d(x,y)\!<\!R,\] 
	for some $R>0$, and $\Omega$ is $R$-connected. 
	
	 Then for every nontrivial  $u_0\ge 0$  in $X$, the solution 
	$u(t)$ of \eqref{eq1} is strictly positive, for all $t>0$. 
\end{teo}
\begin{proof}
	Thanks to Proposition \ref{positive_solution_with_u_0_positive}, we know that 
	$u\ge 0$ in $\Omega$,  for all $t\ge 0$. 
	We take $$v(t)=e^{h(\cdot)t}u(t),$$ 
	then recalling the definition of the essential  support in Definition \ref
	{definition_positive_set}, we have $P(u(t))=P(v(t))$, for all $t\ge 0$. 
	From the argument above \eqref{Form_Var_Const}, we know that $v$ satisfies 
	\begin{equation}\label{strong_max_prin_001}
	v_t(t)=e^{h(\cdot)t}K_J(u(t))\ge 0,\quad\forall t\ge 0.
	\end{equation}
	Integrating \eqref{strong_max_prin_001} in $[s,\,t]$,  we obtain
	\begin{equation}\label{v_t_en_funcion_v_s}
	v(t)=v(s)+\displaystyle\int_{s}^t v_t(r)dr\ge v(s), \;\;\mbox{ for any } \, t\ge s\ge 0.
	\end{equation}
	Then $P(v(t))\supset P(v(s))$, $\forall t\ge s$. Moreover, since $v(t)
	=e^{h(\cdot)t}u(t)$ and  thanks to \eqref{v_t_en_funcion_v_s}, we obtain
	\[
	u(t)\ge e^{-h(\cdot)(t-s)}u(s).
	\]
	This implies that $P(u(t))\supset P(u(s))$, $\forall t\ge s$. As a consequence of  
	\eqref{v_t_en_funcion_v_s}, we have that for any subset  $D\subset \Omega$,
	\begin{equation}\label{v_restring_D}
	\left.v\right|_D(t)=\left.v\right|_D(s)+\displaystyle\int_{s}^t \left.\left(e^{h(\cdot)r}K_J
	(u(r))\right)\right|_D dr.
	\end{equation}
	Since $P(v(t))\supset P(v(s))$ for all $ t\ge s$, and from \eqref{v_restring_D}, 
	we have that  
	\begin{equation}\label{v_restring_D_2}
	P(u(t))\cap D=P(v(t))\cap D\supset P(K_Ju (r))\cap D,\quad \mbox{for all }\;r\in[s,t].
	\end{equation}
	Moreover, applying Proposition \ref{K_aumenta_positividad} to $u(s)$, we have 
	\begin{equation}\label{v_restring_D_3}
		P(K_Ju (r))\supset P(K_J(u(s)))\supset 
		P^1_R(u(s))=\!\!\!\!\!\!\bigcup\limits_{x\in P(u(s))}\!\!\!\! B(x,R)\,\; 
		\mbox{for all }\;r\in [s,\, t].
	\end{equation}
	Hence, if we consider the set $D=P^1_R(u(s))$. From 
	\eqref{v_restring_D_2} and \eqref{v_restring_D_3}, we have that 
	\begin{equation}\label{pos_increases_ball_R}
		P(u(t))\supset P_R^1(u(s)),\quad \mbox{for all }\;t> s.
	\end{equation}
	Hence the essential support of the solution at time $t$, contains the balls 
	of radius $R$ centered at the points in the support of the solution at time $s<t$.  
	
	We fix $t>0$, and let $\mathcal{C}\subset \Omega$ be a 
	 compact set, then  
	Proposition \ref{K_aumenta_positividad} implies that 
	exists $n_0\in \mathbb{N}$, such that $ \mathcal{C}\subset P^{\,n}(u_0)$ for all 
	$n\ge n_0$. 
	We consider the sequence of times  
	$$t=t_n,\; t_{n-1}=t(n-1)/n,\, ...,t_j=t\, j/n,\, ... \,,\, t_1=t/n,\, t_0=0. $$ 
	Therefore, thanks to \eqref{pos_increases_ball_R}, we have that  the 
	essential supports at time $t$, contains the balls of radius $R$ 
	centered at the points in the essential support at time $t_{n-1}$,  $P_R^1(u(t_
	{n-1}))$, which contains the balls of radius $R$ centered at the points 
	in the essential support at time $t_{n-2}$, then  
	$P_R^2(u(t_{n-2}))$. Hence repeating this argument, we have 
	\[
	\begin{array}{l}
	P(u(t))=P(u(t_n)) \supset P_R^1(u(t_{n-1}))\supset P_R^2(u(t_{n-2}))\supset
	\ldots\supset  P_R^n(u_0)\supset \mathcal{C}.
	\end{array}
	\]	
	Thus, we have proved that $u(t)$ is strictly positive for every compact set in 
	$\Omega,\;\forall t>0$. Therefore, $u(t)$ is strictly positive in $\Omega$, 
	for all $t>0$.
\end{proof}

\begin{corolario}\label{flux_non_symetric}
	Under the assumptions of Theorem \ref{linear_solution_strictly_pos}, 
	if  $u_0\ge 0$, not identically zero, with $P(u_0)\ne \Omega$, then the 
	solution to \eqref{eq1} has to be sign changing in $\Omega,\;\forall t<0$.
\end{corolario}
\begin{proof}
	 We argue by contradiction.  Let us assume first that there exists 
	 $t_0<0$ such  that $u(\cdot,t_0)\equiv 0$. We take $u
	 (\cdot,t_0)$ as initial data,  then  solving forward in time $u(\cdot,t)
	 \equiv 0$, for all $t\ge t_0$.  Hence, we arrive to contradiction, 
	 and $u(t_0)$ is not identically zero. 
	 
	 Secondly, let us assume that there exists ~$t_0<0$~ such that 
	 ~$u(\cdot,t_0)\le0$, not identically zero. We take  ~$-u(\cdot,t_0)\ge0$ 
	 as the initial data, then thanks to  Theorem 
	 \ref{linear_solution_strictly_pos}, the solution to \eqref{eq1}, satisfies 
	 that  $u(x,0)<0,\;\forall x\in\Omega$. Thus, we arrive to contradiction.
	 
	 Now, we assume that there exists ~$t_0<0$~ such that ~$u
	 (x,t_0)\ge0$. 
	 Let $u(\cdot,t_0)\ge0$ be the initial data, then thanks to Theorem 
	 \ref{linear_solution_strictly_pos}, the solution to \eqref{eq1}, satisfies that  
	 $u(x,0)>0,\;\forall x\in\Omega$. Thus, we arrive to
	 contradiction. 
	 
	 Therefore, the solution has to be sign changing for all 
	 negative times.
\end{proof}


\subsection{Asymptotic regularizing effects}\label{Regularizing effects}

In general, the group associated to \eqref{eq1} has no regularizing effects.  
However, we will prove that there exists a part 
of the group, that we call $S_2(t)$ 
that is compact, so it somehow regularizes. Moreover, there exists another 
part of the group that we call $S_1(t)$ which does not regularize, i.e., it 
carries the singularities of the initial data, but it decays to zero exponentially 
as $t$ goes 
to $\infty$, if $h\ge 0$. Thus, we will have a 
 regularizing effect when $t$ goes to $\infty$, that is, asymptotic
 smoothness according to \cite[p. 4]{Hale}.

\begin{teo}\label{regularization}
	Let  
	$\mu(\Omega)<\infty$. 
	For $1\le p\le  q\le\infty$, let $X=L^q(\Omega)\,\mbox{  or  }\,\mathcal
	{C}_b(\Omega)$.
	If $K_J\in\mathcal{L}(L^p(\Omega), X)$ is compact,  
	(see Proposition \ref
	{K_compact}), 
	and $h$ satisfies 
	$$h(x)\ge \alpha>0\;\mbox{ for all }\;x\in\Omega,$$  
	and $u_0\in L^p(\Omega)$, then the group associated to the problem 
	\eqref{eq1},  satisfies that
	\[u(t)=S_{K,h}(t)u_0=S_{1}(t)u_0+S_{2}(t)u_0\]
	 with 
	 \renewcommand{\labelenumi}{\roman{enumi}.}
	 \begin{enumerate}
	 \item $S_{1}(t)\in \mathcal{L}(L^p(\Omega))\;\;\forall t>0$, and 
	 $\;\|S_{1}(t)\|_{\mathcal{L}(L^p(\Omega), L^p(\Omega))}\rightarrow 0$ 
	 exponentially, as $t$ goes to $\infty$.
	 \item  $S_{2}(t)\in\mathcal{L}(L^p(\Omega), X)$ is compact,  
	 $\forall t>0$.
	 \end{enumerate}
	 Therefore $S_{K,h}(t)$ is asymptotically smooth.
\end{teo}
\begin{proof}
	We write the solution associated to \eqref{eq1}, as in \eqref{Form_Var_Const}, then 
	we have that 
	 $$u(x,t)=S_{K,h}(t)u_0(x) = e^{-h(x)t}u_0(x)+\displaystyle\int_0^t 
	 e^{-h(x)(t-s)}K_Ju (x,s)ds,\;\forall x\in\Omega$$ 
	 and we define $S_{1}(t)u_0=e^{-h(\cdot)t}u_0$,
         $S_{2}(t)u_0=\int_0^t e^{-h(\cdot)(t-s)}K_Ju
         (s)ds$.  

\noindent 	{\it  i. $\quad$}Since $u_0\in L^p(\Omega)$ and $h\in L^{\infty}(\Omega)$ 
	with $h\ge\alpha>0$, then $\,S_{1}(t)u_0=e^{-h(\cdot)t}u_0\in L^p(\Omega)$ 
	and
	\[
	\begin{array}{ll}
	\|S_{1}(t)\, u_0\|_{L^p(\Omega)} = \|e^{-h(\cdot)t}u_0(\cdot)\|_{L^p(\Omega)} 
	\le e^{-\alpha t} \|u_0\|_{L^p(\Omega)}.
	\end{array} 
	\]  
	{\it  ii. $\quad$} Fix $t>0$, as $h\in L^{\infty}(\Omega)$, $S_{K,h}(s)\in 
	\mathcal{L}( L^p(\Omega))\;\forall s\in[0,t]$, and 
	$K_J\in\mathcal{L}(L^p(\Omega), X)$, then 
        \begin{displaymath}
          	\|S_{2}(t)(u_0)\|_{X} \le e^{-\alpha t} \int_0^t \|K_J(S_{K,h}(s)u_0)\|_{X}ds
			 \le e^{-\alpha t }t\max\limits_{0\le s\le  t} \|K_J(S_{K,h}(s)u_0)\|_{X}<\infty. 
        \end{displaymath}
	
	Let us see now that $S_{2}(t)\in \mathcal{L}(L^p(\Omega), X)$ is compact 
	$\forall t>0$. Fix $t>0$ and consider a bounded set $\mathcal{B}$ of initial data. 
	We denote $ S_2(t)u_0=\int_{0}^t F_{u_0}(s)ds$, with 
	$$F_{u_0}(s)=e^{-h(\cdot)(t-s)}K_J(S_{K,h}(s)u_0).$$
	Assume we have proved that 
	$F_{u_0}(s)\in\!\!\!\text{{\calligra C}}\;\;$, where {\calligra C  }  is a 
	compact set in $X$,   for all $s\in[0, t]$ and for all 
	$u_0\in\mathcal{B}$. Then we have that 
	$\frac{1}{t}S_2(t)(u_0)\in \overline{co}\big(\!\!\text{{\calligra C  }}\,\big),\;
	\forall u_0\in\mathcal{B}$, 
	and thanks to  Mazur's Theorem, we obtain that 
	$\frac{1}{t}S_2(t)(\mathcal{B})$ is in a compact set of $X$. 
	Therefore $S_2(t)$ is compact. 
	Now, we have to prove that $F_{u_0}(s)=e^{-h(\cdot)(t-s)}
	K_J(S_{K,h}(s)u_0)$ belongs to a compact set, for all 
	$(s,\,u_0)\in[0,t]\times\mathcal{B}$.
	
	First of all, we check that $K_J(S_{K,h}(s)u_0)$ belongs to 
	a compact set 
	$\mathcal{W}$ in $X$, for all $(s,\,u_0)\in[0,t]\times\mathcal{B}$. 
	Since $K_J$ is compact, we just have to prove that the set 
	\[B=\{S_{K,h}(s)u_0:\, (s,\,u_0)\in[0,t]\times\mathcal{B}\}\] 
	is bounded. In fact, since $K_J-hI\in\mathcal{L}(L^p(\Omega), 
	L^p(\Omega))$, then   for some $\delta >0$ 
	\[
	\begin{array}{l}
	\|S_{K,h}(s)u_0\|_{L^p(\Omega)}=\|u(\cdot,s)\|_{L^p(\Omega)}
	\le Ce^{\delta s}\|u_0\|_{L^p(\Omega)}
	\le Ce^{\delta t}\|u_0\|_{L^p(\Omega)},
	\end{array}
	\]
	 for all $(s,\,u_0)\in[0,t]\times\mathcal{B}$. 
	 Then, since $\mathcal{B}$ is bounded, we obtain that 
	 $B$ is bounded in $L^p(\Omega)$.
	 
	 Finally, we just need to prove that $F_{u_0}(s)$ is in a compact 
	 set for all $(s,\,u_0)\in[0,t]\times\mathcal{B}$. Since the mapping 
	\[
	\begin{array}{lrcl}
	M: & [0,\,t]\times X&\longrightarrow &X\\
	&(s,f)&\longmapsto&e^{-h\,(t-s)}f
	\end{array}
	\]
	is continuous, then $M$ sends the compact set 
	$[0,\,t]\times\mathcal{W}$ into a compact set {\calligra{C  }}. Thus,  
	$F_{u_0}(s)$ belongs to a compact set, {\calligra{C  }}, 
	$\forall(s,\,u_0)\in[0,t]\times\mathcal{B}$. 
\end{proof}

\subsection{The Riesz projection and asymptotic behavior}
 \noindent

In this section we study the asymptotic behavior of the solution of the 
problem \eqref{eq1} by using the Riesz projection, which is given in terms 
of the spectrum of the operator. Since the spectrum of the operator 
$L=K_J-hI
$ has been proved in Proposition \ref{independent_spectrum_K_h} 
to be independent of $X=L^p(\Omega)$, with $1\le p\le\infty$ or $X=
\mathcal{C}_b(\Omega)$, then the asymptotic behavior of 
the solution of \eqref{eq1} will be characterized 
with the Riesz projection that can be explicitly computed in
$L^2(\Omega)$.

We now briefly recall the construction of the Riesz projection, for
more details see \cite[chap. 1]{Gohberg} and Section III.6.4 in  \cite{kato}. Consider an operator $L\in
\mathcal{L}(X, X)$, where $X$ is a Banach space
and consider the linear problem 
  \begin{equation}\label{problema_L}
	 \left\{
	 \begin{array}{rl}
	 u_t(x,t) & =Lu (x,t)\\
	 u(x,0) & =u_0(x),\mbox{   with } u_0\in X . 
	 \end{array}
	 \right.
 \end{equation} 
Since $L$ is a bounded operator, then $Re(\sigma(L))\le\delta$, and the 
norm of the semigroup satisfies that 
\begin{equation}\label{cota}
||e^{Lt}||_{\mathcal{L}(X)}\le C_0e
^{(\delta+\varepsilon) t},\quad  \quad \;t\ge 0.
\end{equation}

Then, given an isolated part $\sigma_1$ of $\sigma(L)$ we define the {\bf Riesz projection} of $L$ 
corresponding to the isolated part $\sigma_1$, $Q_
{\sigma_1}$, as the  bounded linear operator on $X$ given by  
\[Q_{\sigma_1}=\frac{1}{2\pi i}\displaystyle\int_{\Gamma}(\lambda I-L)^{-1}d
\lambda,\]
where $\Gamma$ consists of a finite number of rectifiable Jordan curves, 
oriented in the positive sense around $\sigma_1$, separating $\sigma_1$ 
from $\sigma_2=\sigma(L)\setminus\sigma_1$. This means that $\sigma_1$ belongs to the inner region of 
$\Gamma$, and $\sigma_2$ belongs to the outer region of $\Gamma$. The 
operator $Q_{\sigma_1}$ is independent of the 
path $\Gamma$ described as above.

Assume the spectrum of $L$ is the disjoint union of two non-empty closed subsets $
\sigma_1$ and $\sigma_2$. To this decomposition of the spectrum 
corresponds a direct sum decomposition of the space, $X=U\oplus V$, such that $U$ 
and $V$ are $L-$ invariant subspaces of $X$, the spectrum of the 
restriction $L|U$ is equal to $\sigma_1$ and that of $L|V$ to $\sigma_2$.
If we  assume that
$$\delta_2<\mbox{Re}(\sigma_1)\le \delta_1,\;\mbox{Re}(\sigma_2)\le \delta_2\,,\;\mbox{ with }
\; \delta_2<\delta_1,$$
%
then we  have that the solution to \eqref{problema_L}, can 
be written as 
\[u(t)=Q_{\sigma_1}(u)(t)+Q_{\sigma_2}(u)(t).\]
On the other hand, the solution of \eqref{problema_L} is equal to $\;u(t)=e^{Lt}u_0$. 
Thus, thanks to \eqref{cota} and since $\;\mbox{Re}
(\sigma_2)\le \delta_2\,$ we obtain that for $t>0$
%
\begin{displaymath} 
  \|Q_{\sigma_2}(u(t))\|_X=\|\left(Q_{\sigma_2}\circ e^{L\,t}\right)u_0\|_X
\smallskip  =\|e^{L_2 t}\,Q_{\sigma_2}(u_0)\|_X
\smallskip
\le C_2\, e^{(\delta_2+\varepsilon)t}\|Q_{\sigma_2}(u_0)\|_X, 
\end{displaymath}
where, $L_2=L|\mbox{Im}\,Q_{\sigma_2}$. 
 With this, we get the following, see \cite[chap. 1]{Gohberg}. 

 \begin{teo}\label{Asymptotic behavior}
	Consider $L\in\mathcal{L}(X)\;$ and let $\;\sigma(L)$ be a disjoint 
	union of two closed subsets $
	 \sigma_1$ and $\sigma_2$, with
	 $\delta_2<\mbox{Re}(\sigma_1)\le \delta_1$, $\;\mbox{Re}
	 (\sigma_2)\le \delta_2\,$, with 
	 $\,\delta_2<\delta_1$. Then the solution 
	 of \eqref{problema_L} satisfies 
	\[\lim\limits_{t\rightarrow \infty} \|e^{-\mu\, t}\big(u(t)-Q_{\sigma_1}(u)(t)
	\big)\|_X=0,\;\quad \forall\mu>\delta_2.\]
\end{teo}

The assumptions of the  following proposition, are tailored for the
case $L=K_J-hI$ in (\ref{eq1}) and allows to compute the Riesz
projection in terms of the Hilbert projection. 

 \begin{prop}\label{projection_spectro_indep_p}
	For $1\le p_0<p_1\le\infty$, with $2\in
	 [p_0,\,p_1]$, let $X=L^p(\Omega)$, with $p\in[p_0,p_1]$, or $X=
	 \mathcal{C}_b(\Omega)$. We assume $L\in\mathcal{L}(X, X)$ is 
	 selfadjoint in $L^2(\Omega)$, the spectrum of $L$, 
	 $\sigma_{X}(L)$, 
	 is independent of $X$, and the largest eigenvalue associated to $L$, 
	 $ \lambda_1$ is  simple and isolated, with associated  
	 eigenfunction $\Phi_1\in L^p(\Omega)\cap L^{p'}(\Omega)$, for   
	 $p\in [p_0,p_1]$, if $X=L^p(\Omega)$, or  $\Phi_1\in \mathcal{C}_b(\Omega)
	 \cap L^1 (\Omega)$, if $X=\mathcal{C}_b(\Omega)$, and $\|\Phi_1\|_{L^2(\Omega)}
	 =1$. 
	 
	 If $\sigma_1=\{\lambda_1\}$, and $\Gamma$ is  the curve around 
	 only $\lambda_1$, 
	 then for $u \in X$, the Riesz projection  associated to 
	 $\sigma_1$ is given by 
         \begin{equation} \label{eq:Riesz=Hilbert}
           Q_{\sigma_1}(u)= \big(\int_{\Omega}u\,\Phi_1 \big),\Phi_1   . 
         \end{equation}
 \end{prop}
 \begin{proof}	
   First, working in $L^{2}(\Omega)$, it is well known that the Riesz
   projection coincides with the Hilbert projection, that is
   (\ref{eq:Riesz=Hilbert}) holds for all $u\in L^2(\Omega)$; see from
   Sections III.6.4 and III.6.5 in \cite{kato}.

         Now in $X=L^p(\Omega)$ for $ p\in [p_0,p_1]$ or $X=
	  \mathcal{C}_b(\Omega)$, since the
         spectrum, $\sigma_{X}(L)$, is independent of $X$, we have
         that  	 the projection $P(u)=\langle u,\,
	 \Phi_1\rangle \Phi_1$ is well defined for $u\in X$ because by 
	 hypothesis, $\Phi_1\in L^{p'}(\Omega)\cap L^p(\Omega)$ for all $p\in
	 [p_0,p_1]$, if $X=L^p(\Omega)$, 
	  or $\Phi_1\in \mathcal{C}_b(\Omega)\cap L^1(\Omega)$, if $X=
	  \mathcal{C}_b(\Omega)$. In fact $P \in
          \mathcal{L}(X, X)$. On the other hand, since the set 
	 $$V=span\left[\,\chi_D\,;\, D\subset\Omega\;\;\mbox
	{with }\;\mu(D)<\infty\right] \subset L^{2}(\Omega),$$
	 where $\chi_D$ is the characteristic function of $D\subset
         \Omega$, 
         is dense in $L^p(\Omega)$, and $Q_{\sigma_1}\equiv P$ 
	 in $V$, they coincide in $X=L^{p}(\Omega)$. Finally for $X=
         \mathcal{C}_b(\Omega)$, we use that $L^2(\Omega)\cap
         \mathcal{C}_b(\Omega)$ is dense in $\mathcal{C}_b 
	 (\Omega)$ and again  $Q_{\sigma_1}\equiv P$   in
         $X$. 
 \end{proof}

\subsection{Asymptotic behaviour of the solution of the nonlocal diffusion  
problem}

Let $(\Omega,\mu,d)$ be a metric measure space with $\Omega$ compact. 
In this section be apply the results of the previous section about the asymptotic 
behavior of the solution for the problem 
\begin{displaymath} 
	 \left\{
	 \begin{array}{lll}
	 u_t(x,t) & =(K_J-hI)u (x,t), & x\in\Omega,\; t>0,\\
	 u(x,0) & =u_0(x), & \mbox{with } u_0\in X.
	 \end{array}
	 \right.
 \end{displaymath}

 We study two problems to which we apply the results of the previous sections. In 
particular we consider the cases where $h$ constant  or
$h=h_0=\displaystyle\int_{\Omega}J(\cdot,y)dy$, with $J\in
L^{\infty}(\Omega, L^1(\Omega))$.   

 
 \noindent\textsc{\bf Case $h$ constant.} \ 
For $h=a\in\mathbb{R}$ constant we have the problem
\begin{equation}\label{dirichlet}
 \left\{
 \begin{array}{ll}
 u_t(x,t) & =(K_J-aI)u (x,t), \\
 u(x,0) & =u_0(x) \in L^p(\Omega).
 \end{array}
 \right.
 \end{equation}
 

Then we have. 

\begin{prop}\label{asymp_beh_Dirichlet}
	Let 
	$\Omega$ be compact and connected. Let $X=L^p(\Omega)$, with 
	$1\le p\le\infty$, or $X=\mathcal{C}_b(\Omega)$.
	Let $K_J\in \mathcal{L}(L^{1}(\Omega),\mathcal{C}_b(\Omega))$ 
	be compact, (see Proposition \ref{K_compact})  and 
	assume  $J(x,y)=J(y,x)
	$ with 
	 \begin{displaymath}  
	 	J(x,y)>0,\;\forall x,\,y\in\Omega\; \mbox{such that}\; 
		d(x,y)<R,\; \mbox{ for some } \;R>0.
	\end{displaymath}
	Then the solution $u$ of \eqref{dirichlet} satisfies that
	\begin{displaymath} 
		\lim\limits_{t\rightarrow \infty} \|e^{-\lambda_1t}u(t)-C^*\Phi_1\|_{X}
		=0,
	\end{displaymath}
	where $\displaystyle C^*= \int_{\Omega}u_0\Phi_1$, and $\Phi_1$ is an eigenfunction associated to 
	$\lambda_1$, normalized in $L^{2}(\Omega)$. 
\end{prop}
\begin{proof} 
	 From Proposition \ref{sigma_independiente}, we have that  
	 $\sigma_{X}(K_J)$ is independent of $X$. Moreover, since $J(x,y)=J(y,x)$, 
	 then from  Proposition \ref{sigma(K)}, we know that  $\sigma(K_J)\setminus 
	 \{0\}$ is a real sequence of eigenvalues $\left\{\mu_n\right\}_{n\in\mathbb{N}}
	 $ of  finite multiplicity that converges to $0$. 
	 Furthermore, the hypotheses of Proposition 
	 \ref{spectral_radius_K} are satisfied, then  the largest eigenvalue, $
	 \lambda_1=r(K_J)$, is an isolated  simple eigenvalue, and the eigenfunction 
	 $\Phi_1\in\mathcal{C}_b(\Omega)$ associated to it, can be taken positive. Since the spectrum does not depend 
	 on $X$, we have that, $\Phi_1\in X$, 
	in particular $\Phi_1\in L^p(\Omega)\cap L^{p'}(\Omega)$, and $\Phi_1\in 
	\mathcal{C}_b(\Omega)\cap L^{1}(\Omega)$. Then the spectrum of
        $K_J-aI$ is  $\left\{\lambda_n = \mu_n - a \right\}_{n\in\mathbb{N}}$ and $\Phi_1$ is a positive 
	 eigenfunction associated to $\lambda_1$.

	 

	Thus, for $u_0\in X$ thanks to Theorem \ref
	{Asymptotic behavior}, the solution of \eqref
	{dirichlet} satisfies
	\begin{displaymath}  
	\lim\limits_{t\rightarrow \infty} \|e^{-\lambda_1t}\big( u(t)-Q_{\sigma_1}(u)(t)
	\big)\|_
	{X}=0.
	\end{displaymath}
and by Proposition  \ref{projection_spectro_indep_p} we have
$Q_{\sigma_1}=P$. 
%
Thus, 	since $u(x,t)=e^{(K_J-aI)t}u_0(x)$, 	 we have that 
	\begin{displaymath}  
		Q_{\sigma_1}(u)(t)=Q_{\sigma_1}(e^{(K_J-aI)t}u_0)=e^{(K_J-aI)t}
		Q_{\sigma_1}(u_0) = C^{*} e^{(K_J-aI)t} \Phi_1 =
                C^{*} e^{\lambda_1t}\Phi_1.
	\end{displaymath}
where $
C^*= \int_{\Omega}u_0 \Phi_1$,
and we get the result. 
\end{proof}
%
%
%



 \noindent\textsc{\bf Case $h=h_0\in L^{\infty}(\Omega)$. } Assume
 $J\in L^{\infty}(\Omega,  L^1(\Omega))$  and  consider  
the problem
\begin{equation}\label{neumann}
 \left\{
 \begin{array}{ll}
 u_t(x,t) & =(K_J-h_0I)u (x,t)\\
 u(x,0) & =u_0(x),\mbox{   with } u_0\in L^p(\Omega)
 \end{array}
 \right.
 \end{equation}
 
 In the following proposition, we prove that the solution of \eqref{neumann} 
 goes  exponentially in norm $X$ to the mean value in $\Omega$ of the initial data.
 
\begin{prop}\label{asymp_beh_Neumann}
	Let 
	$\mu(\Omega)<\infty$, let $X=L^p(\Omega)$, with $1\le p\le\infty$ or 
	$X=\mathcal{C}_b(\Omega)$. We assume 
	 $K_J\in \mathcal{L}(L^{1}(\Omega),\mathcal{C}_b(\Omega))
	$ is compact, (see Proposition \ref{K_compact}) and  $J$ 
	satisfies $J\in L^{\infty}(\Omega, L^1(\Omega))$,  
	$J(x,y)=J(y,x)$ and
	 \begin{displaymath}  
	 	J(x,y)>0,\;\forall x,\,y\in\Omega\; \mbox{such that}\; 
		d(x,y)<R,\; \mbox{for some} \;R>0.
	\end{displaymath}
	We also assume that $h_0(x)>\alpha>0$, for all $x\in\Omega$. 
	
	Then the solution $u$ of \eqref{neumann} satisfies that
	\begin{displaymath}  
		\lim\limits_{t\rightarrow \infty} \left\|e^{\beta t}\left( u(t)-
		\displaystyle\frac{1}{\mu(\Omega)}\int_{\Omega}u_0 \right)\right\|_
		{X}	=0,
	\end{displaymath}
	for some  $\beta>0$.  
\end{prop}
\begin{proof}
	Since  $K_J\in\mathcal{L}(L^{1}(\Omega),\mathcal{C}_b(\Omega))$ is compact, then 
	$K_J\in\mathcal{L}(X, X)$ is compact. Thanks to Theorem \ref{espectro_K-h}, we know that 
	 $$\sigma_X(K_J-h_0I)=\overline{R(-h_0)}\cup\left\{\mu_n\right\}_{n=1}^M,\quad \mbox
	{with } M\in\mathbb{N}\;\mbox{ or } M=\infty.$$ 
	If $M=\infty$, then $\{\mu_n\}_{n=1}^{\infty}$ is a sequence of eigenvalues of $K_J-h_0I$ 
	with finite multiplicity, that has accumulation points in $R(-h)$. Moreover, from  Proposition 
	\ref{independent_spectrum_K_h}, $\sigma_{X}(K_J-h_0I)$ is
        independent of $X$. Also, from  Corollary
        \ref{spectrum_negative}, we have that $\sigma_{X}(K_J- 	 h_0I)\le 0$, 
	 and $0$ is an isolated simple eigenvalue with only constant
         eigenfunctions. 
	 	 
	 Moreover, since $J(x,y)=J(y,x)$ and thanks to Proposition \ref{adjoint_operator}, 
	 $ K_J-h_0I$ is selfadjoint in $L^2(\Omega)$, thus, $	\{\mu_n\}\subset
	 \mathbb{R}$. 
	Hence, we consider  $\sigma_1=\{0\}$ an isolated part of $\sigma(K_J-
	h_0I)$, with associated 
	eigenfunction $\Phi_1=1/\mu(\Omega)^{1/2}$, and $\sigma_2=\sigma
	(K_J-h_0I)
	\setminus \{0\}$. Then thanks to Theorem \ref{Asymptotic
          behavior},  
	\begin{displaymath}  
		\lim\limits_{t\rightarrow \infty} \left\|e^{\beta t}
		\left( u(t)-Q_{\sigma_1}(u)(t)\right)\right\|_{X}=0,
	\end{displaymath}
for some  $\beta >0$ and by Proposition \ref{projection_spectro_indep_p}, and since  
        $Q_{\sigma_1}=P$. Since   $u(x,t)=e^{(K_J-h_0I)t}u_0(x)$, 
	\begin{displaymath} 
	\begin{array}{ll}
		Q_{\sigma_1}(u)(t)& =Q_{\sigma_1}(e^{(K_J-h_0I)t}
		u_0) = e^{(K_J-h_0I)t} Q_{\sigma_1}(u_0) 
		\smallskip\\&\displaystyle= \left(
                \int_{\Omega} u_0\Phi_1\right)  e^{(K_J-h_0I)t} \Phi_1
              = \big( \int_{\Omega} u_0\Phi_1 \big) \Phi_1 =
              \frac{1}{\mu(\Omega)}\int_{\Omega}u_0 . 
	\end{array}
	\end{displaymath}
\end{proof}

\begin{remark}
Propositions \ref{asymp_beh_Dirichlet} and \ref
{asymp_beh_Neumann} where proven    in \cite{Rossi}, in the case where
 $\Omega$ is an open set in $\R^{N}$ and for $X= L^2(\Omega)$ or $X=
\mathcal{C}(\overline{\Omega})$. 
\end{remark}

\addcontentsline{toc}{section}{Bibliography}

\end{document}